\documentclass[11pt]{article}

\usepackage[margin=1.1in]{geometry}  
\usepackage{graphicx}               
\usepackage{amsmath}              
\usepackage{amsfonts}              
\usepackage{bm}                        
\usepackage{amsthm}                
\usepackage{mathtools}             
\usepackage{esint}                     
\usepackage{color}
\usepackage{amssymb}             
\usepackage[hidelinks]{hyperref}
\usepackage[noabbrev]{cleveref}
\usepackage{authblk}                 
\usepackage{tikz}                       
\usepackage{pgfplots}                
\usetikzlibrary{patterns}              

\newtheorem{thm}{Theorem}[section]
\newtheorem{lem}[thm]{Lemma}
\newtheorem{prop}[thm]{Proposition}
\newtheorem{cor}[thm]{Corollary}

\newtheorem{rmk}[thm]{Remark}
\newtheorem{definition}[thm]{Definition}

\DeclareMathOperator{\id}{Id}

\DeclareMathOperator{\loc}{loc}

\DeclareMathOperator{\dist}{dist}

\DeclareMathOperator{\dive}{div}

\newcommand{\RR}{\mathbb{R}}     
      
\newcommand{\NN}{\mathbb{N}}     
\newcommand{\F}{\mathcal{F}}  
\newcommand{\J}{\mathcal{J}}  
\newcommand{\K}{\mathcal{K}}

\newcommand{\e}{\varepsilon} 
\newcommand{\ga}{\gamma}
\newcommand{\la}{\lambda}  
          
\numberwithin{equation}{section} 

\begin{document}


\title{On the behavior of the free boundary for a one-phase Bernoulli problem with mixed boundary conditions}

\author{
    Giovanni Gravina\\
    Department of Mathematical Analysis\\
    Faculty of Mathematics and Physics \\
    Charles University\\
    Prague, Czech Republic\\
    gravina@karlin.mff.cuni.cz
  \and
    Giovanni Leoni\\
    Department of Mathematical Sciences\\
    Carnegie Mellon University\\
    Pittsburgh, PA, USA\\
    giovanni@andrew.cmu.edu
    }

\maketitle
\begin{abstract}
This paper is concerned with the study of the behavior of the free boundary for a class of solutions to a one-phase Bernoulli free boundary problem with mixed periodic-Dirichlet boundary conditions. It is shown that if the free boundary of a symmetric local minimizer approaches the point where the two different conditions meet, then it must do so at an angle of $\pi/2$. 
\vspace{.3cm}
\vspace{.2cm}
\end{abstract}


\section{Introduction}
In the seminal paper \cite{MR618549}, Alt and Caffarelli considered the minimization problem for the functional 
\begin{equation}
\label{JAC}
J(u) \coloneqq \int_{\Omega}\left(|\nabla u|^2 + \chi_{\{u > 0\}}Q^2\right)\,d\bm{x},
\end{equation}
defined over the class 
\[
K \coloneqq \left\{u \in L^1_{\loc}(\Omega) : \nabla u \in L^2(\Omega;\RR^N) \text{ and } u = u_0 \text{ on } \Gamma \right\}.
\]
Here $\Omega$ is an open connected subset of $\RR^N$, $\partial \Omega$ is Lipschitz continuous, $\Gamma \subset \partial \Omega$ is a measurable set with $\mathcal{H}^{N - 1}(\Gamma) > 0$, $u_0$ a nonnegative function in $L^1_{\loc}(\Omega)$ such that $J(u_0) < \infty$, and $Q$ is a measurable function satisfying
\begin{equation}
0 < q_{\min} \le Q(\bm{x}) \le q_{\max} < \infty 
\end{equation} 
for all $\bm{x} \in \Omega$. Alt and Caffarelli proved the existence of global minimizers and showed that local energy minimizers are nonnegative, locally Lipschitz continuous in $\Omega$, and harmonic in the set $\{ u > 0 \}$. Furthermore, they proved that if $Q \in C^{k,\alpha}$ (resp. analytic), then the free boundary $\partial \{u > 0\}$ of local energy minimizers is locally a curve of class $C^{k + 1, \alpha}$ (resp. analytic) in $\Omega$ provided $N = 2$, while in dimension $N \ge 3$ the reduced free boundary $\partial^{\operatorname{red}}\{ u > 0\}$ is locally a hypersurface of class $C^{k + 1, \alpha}$ (resp. analytic) in $\Omega$ and the singular set
\[
\Sigma_{\operatorname{sing}} \coloneqq \Omega \cap \left\{\partial \{ u > 0\} \setminus \partial^{\operatorname{red}}\{ u > 0\}\right\}
\]
has zero $\mathcal{H}^{N - 1}$ measure. We refer to the papers \cite{MR2082392},  \cite{MR2572253}, \cite{MR3894044}, \cite{MR3385632}, and \cite{MR1759450} for additional results regarding the interior regularity for the free boundaries of local minima in dimension $N \ge 3$.

In \cite{MR1044809}, Berestycki, Caffarelli, and Nirenberg proved regularity of the free boundary up to the fixed boundary for the zero oblique derivative boundary conditions. Their method is a regularization by singular perturbation, i.e., given the approximate identities $\{\beta_{\e}\}_{\e}$, they considered the family of elliptic equations
\[
Lu_{\e} = \beta_{\e}(u_{\e})
\]
and obtained Lipschitz estimates for the solutions $\{u_\e\}_{\e}$ which are uniform with respect to the regularization parameter and therefore carry over in the limit. Following a similar approach, Gurevich \cite{MR1656068} proved Lipschitz regularity up to the boundary for solutions to one and two phase problems with Dirichlet boundary conditions. We refer the reader to the work of Raynor \cite{MR2475326} for an alternative proof of boundary regularity for Neumann boundary conditions.

The behavior of the free boundary near the fixed boundary was investigated by Karakhanyan, Kenig, and Shahgholian in \cite{MR2267752} where they proved that the free boundary detaches tangentially from a smooth of portion of the fixed boundary where Dirichlet boundary conditions are prescribed (see \Cref{tan}). Their techniques apply to both one and two phase problems. In the recent work \cite{MR3916702}, Chang-Lara and Savin proved that for a certain class of solutions to a one-phase Bernoulli problem, the free boundary is a hypersurface of class $C^{1,1/2}$ in a neighborhood of the smooth Dirichlet fixed boundary by relating its behavior to that of a Signorini-type obstacle problem.

         \begin{figure}
         \begin{center}
         \begin{tikzpicture}[blend group=normal, scale=0.9]
         
         \draw [dashed] (0,1) -- (0.5,1); 
         \draw (0.5,1) -- (9.5,1); 
         \draw [dashed] (9.5,1) -- (10,1);
         
         \draw [red, thick] plot [smooth, tension=0.7] coordinates {(5.4,6.7) (4.3,5) (7.3,1.7) (5, 1)};
         \draw [red, dashed, thick] (5.4,6.7) -- (5.9,6.925);
         \draw [red, thick] (5,1) -- (0.5,1);
         \draw [dashed,domain=0:180] plot ({5 + 4.4*cos(\x)}, {1 + 4.4*sin(\x)});
         
         \draw[fill] (5,1) circle [radius=0.06];
         \node [above] at (2,5) {$u > 0$};      
         \node [above] at (8.5,5) {$u = 0$};   
         \node[above] at (8.5,1) {$\partial \Omega$};
         \node[above] at (2,1) {\textcolor{red}{$\partial \{ u > 0\}$}};
         \end{tikzpicture}
         \caption{The free boundary meets the Dirichlet fixed boundary tangentially.} 
         \label{tan}
         \end{center}
         \end{figure}
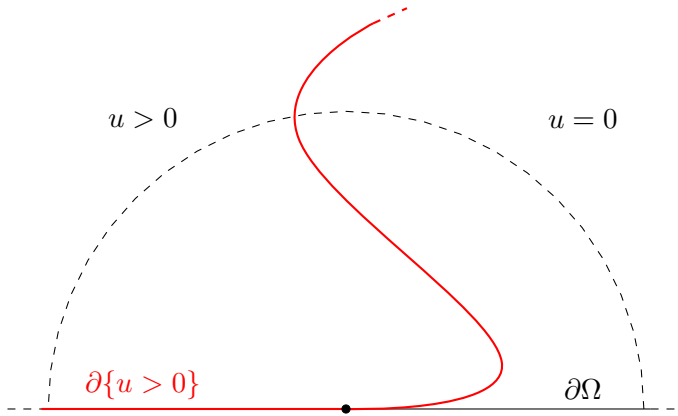

We remark that if $\Gamma$ is sufficiently smooth and $\overline{\Gamma} \neq \partial \Omega$, minimizers of $J$ (see (\ref{JAC})) must satisfy natural boundary conditions on $\Sigma \coloneqq \partial \Omega \setminus \overline{\Gamma}$, i.e., they must solve the mixed Dirichlet-Neumann boundary value problem
\begin{equation}
\label{fbAC}
\left\{
\arraycolsep=1.4pt\def\arraystretch{1.6}
\begin{array}{rlll}
\chi_{\{ u > 0 \}}\Delta u = & 0,& u \ge 0 & \text{ in } \Omega, \\
|\nabla u| = & Q,& u = 0 & \text{ on } \Omega \cap \partial\{u > 0\}, \\
u = & u_0 & & \text{ on } \Gamma, \\
\partial_{\nu}u = & 0 & & \text { on } \Sigma.
\end{array}
\right.
\end{equation}
To our knowledge, the regularity of solutions to (\ref{fbAC}) and the behavior of their free boundaries in a neighborhood of the points where the two different boundary conditions meet is still an open problem. 

In this paper we examine the boundary regularity for a certain class of variational solutions to a one-phase problem near contact points on $\partial \Gamma$ in the simplified situation where $N = 2$ and periodicity conditions are prescribed on $\Sigma$. In addition, we investigate the behavior of the free boundary near such points. 

To be precise, for $m, h, \ga, \la > 0$ we let 
\begin{align}
\label{strip}
\Omega & \coloneqq \left(-\frac{\la}{2},\frac{\la}{2}\right) \times (0, \infty), \\
Q(\bm{x}) & \coloneqq \sqrt{(h - y)_+} \quad \text{ for } \bm{x} = (x,y) \in \RR^2_+, \notag
\end{align}
define the Sobolev space 
\[
H^1_{\operatorname{per}}(\Omega) \coloneqq \left\{u \in H^1_{\loc}(\RR^2_+) : u(x + \la, y) = u(x,y) \text{ for } \mathcal{L}^2 \text{-a.e.\@ } \bm{x} = (x,y) \in \RR^2_+\right\}
\]
consisting of all $\la$-periodic functions in the $x$ variable, and consider the following version of the Alt-Caffarelli functional 
\begin{equation}
\label{Jper}
\J_h(u) \coloneqq \int_{\Omega}\left(|\nabla u|^2 + \chi_{\{u > 0\}}(h - y)_+\right)\,d\bm{x},
\end{equation}
defined for $u$ in the closed convex set
\begin{equation}
\label{Kg}
\K_{\ga} \coloneqq \left\{u \in H^1_{\operatorname{per}}(\Omega) : u = u_0 \text{ on } \Gamma_{\ga}\right\}.
\end{equation}
Here the Dirichlet datum $u_0$, defined by
\begin{equation}
\label{u0}
u_0(x,y) \coloneqq m\left(1 - \frac{y}{\ga}\right)_+,
\end{equation}
is prescribed on 
\[
\Gamma_{\ga} \coloneqq \left(-\frac{\la}{2}, \frac{\la}{2}\right) \times \{0\} \cup \left\{\pm \frac{\la}{2}\right\} \times (\ga, \infty) \subset \partial \Omega.
\]

Throughout the rest of the paper we restrict our attention to the following class of variational solutions. 
\begin{definition}
\label{sm}
Given $u \in \K_{\ga}$, we say that $u$ is a \emph{local minimizer} of the functional $\J_h$ if there exists $\e_0 > 0$ such that $\J_h(u) \le \J_h(v)$ for every $v \in \K_{\ga}$ with 
\[
\|\nabla (u - v)\|_{L^2(\Omega;\RR^2)} + \|\chi_{\{ u > 0\}} - \chi_{\{ v > 0\}}\|_{L^1(\Omega)} \le \e_0.
\] 
Moreover, we say that $u$ is \emph{symmetric} if the following conditions are satisfied:
\begin{itemize}
\item[$(i)$] the support of $u$ in $\Omega$ coincides with its Steiner symmetrization with respect to the line $\{x = 0\}$; 
\item[$(ii)$] $u$ coincides with its symmetric decreasing rearrangement with respect to the variable $x$, i.e., for every $y \in \RR_+$ the map $x \mapsto u(x,y)$ is even, nondecreasing in $(-\la/2,0)$, and nonincreasing in $(0,\la/2)$.
\end{itemize}
\end{definition}
Solutions satisfying conditions $(i), (ii)$ can be observed among global minimizer of $\J_h$ by means of standard symmetrization techniques; for more information, we refer the reader to \cite{GL18} (see also \Cref{exofmin} and \Cref{representative} below).

Our first main result can then be stated as follows.

\begin{thm}
\label{bgl}
Given $m, h, \la > 0$ and $\ga < h$, let $\Omega$, $\J_h$, and $\K_{\ga}$ be defined as in $(\ref{strip})$, $(\ref{Jper})$, and $(\ref{Kg})$, respectively. Let $u \in \K_{\ga}$ be a symmetric local minimizer of $\J_h$ in the sense of \Cref{sm} and assume that $\bm{x}_0 = (-\la/2,\ga)$ is an accumulation point for the free boundary on $\partial \Omega$, i.e.,
\begin{equation}
\label{x0acc}
\bm{x}_0 \in \overline{\partial\{u > 0\} \cap \Omega}.
\end{equation}
Then $\nabla u$ is bounded in a neighborhood of $\bm{x}_0$.
\end{thm}

\Cref{bgl} gives a uniform estimate on the gradient of a symmetric local minimizer in a neighborhood of the point $\bm{x}_0$. We remark that this is the optimal regularity. This kind of result is commonly referred to as a ``bounded gradient lemma'' (see, for example, Lemma 8.1 and 8.2 in \cite{MR682265}, Lemma 2.1 and 2.2 in \cite{MR642623}, and 3.7 Theorem in \cite{MR518852}). Our main contribution is proving that the estimate holds up to the fixed Dirichlet boundary, uniformly with respect to the distance from the point $\bm{x}_0$. This is accomplished through the use of a boundary Harnack principle (see Theorem 11.5 in \cite{MR2145284}). It is important  to observe that the proof of \Cref{bgl} is rather delicate. Indeed, it is well known that solutions of elliptic equations with mixed boundary conditions exhibit a singular behavior near the region where the boundary conditions change (see, e.g., \cite{dauge}, \cite{GL19}, \cite{grisvard85}, \cite{kondratiev}, and \cite{MR0601607})). In our case, the situation is further complicated by the fact that at this point we do not yet know the behavior of the free boundary near $\bm{x}_0$, so none of the standard theory for mixed problems can be applied since it assumes either smooth boundary or a corner.

The following theorem, which is the second main result of the paper, states that the free boundary of a symmetric local minimizer meets the endpoint of the Dirichlet fixed boundary at an angle of $\pi/2$ (see \Cref{perp}).

\begin{thm}
\label{mainthm}
Under the assumptions of \Cref{bgl}, we have that the portion of the free boundary $\partial \{u > 0\}$ in $\{\bm{x} \in \Omega : -\la/2 < x < 0\}$ can be described by the graph of a function $x = g(y)$ and furthermore, the free boundary meets the fixed boundary at the point $\bm{x}_0$ with horizontal tangent, i.e.,
\[
\lim_{y \to \ga}\frac{|g(y) - g(\ga)|}{|y - \ga|} = \infty.
\]
\end{thm}

The importance of \Cref{bgl} is that it allows us to consider blow-up limits. Indeed, as it is often the case for this kind of regularity results (see, for example, \cite{MR2281197}, \cite{MR2237208}, \cite{MR3175292}, \cite{MR1404319}, and \cite{MR2267752}), the proof of \Cref{mainthm} relies heavily on the complete characterization of blow-up solutions (see \Cref{class}). This, in turn, is derived from a monotonicity formula. To be precise, we show that the boundary monotonicity formula of Weiss (see Theorem 3.3 and Corollary 3.4 in \cite{MR2047400}, see also \cite{MR2810856} and  \cite{MR1759450}) holds at the point $\bm{x}_0$ for local minimizers of $\J_h$ in $\K_{\ga}$ with bounded gradients. The main difficulty in the proof of \Cref{mainthm} is that Weiss' results (see Section 4 in \cite{MR2047400}) are restricted to Dirichlet conditions and rely on a non-degeneracy condition of the gradient of the Dirichlet datum. Moreover, his definition of solution is too restrictive for our purposes. Thus adapting his results to the present setting is quite involved.

         \begin{figure}
         \begin{center}
         \begin{tikzpicture}[blend group=normal, scale=0.9]
         
         \draw (3,3) -- (3,9.5);
         \draw [dashed] (3, 9.5) -- (3,10);
         \draw [dashed] (0,1.5) -- (3,1.5);
         \draw (3,1.5) -- (9.5,1.5);
         \draw [dashed] (9.5,1.5) -- (10,1.5);
         
         \draw [red, thick, dashed] (0,4.2) to [out=-10, in = 180] (3,3);
         \draw [red, thick] (3,3) to [out=0, in=190] (6, 4.2) to [out=10, in=270] (7.8,5.7) to [out=90, in= 270] (6.3, 7.6) to [out=90, in=225] (9,9);
         \draw [red, thick, dashed] (9,9) -- (9.5,9.5);
         
         \draw[fill] (3, 3) circle [radius=0.06];
         \node [below] at (3, 2.8) {$(-\la/2, \ga)$};
         \node [above] at (6, 5) {$u = 0$};      
         \node [above] at (6, 3) {$u > 0$};   
         \node [right] at (2, 8.5) {$\partial \Omega$};
         \node [above] at (8.2, 6.5) {\textcolor{red}{$\partial \{ u > 0\}$}};
         \node [above] at (8.2, 1.5) {$y = 0; u \equiv m$};
         \end{tikzpicture}
         \caption{The free boundary hits the point $(-\la/2,\ga)$ at an angle of $\pi/2$.} 
         \label{perp}
         \end{center}
         \end{figure}
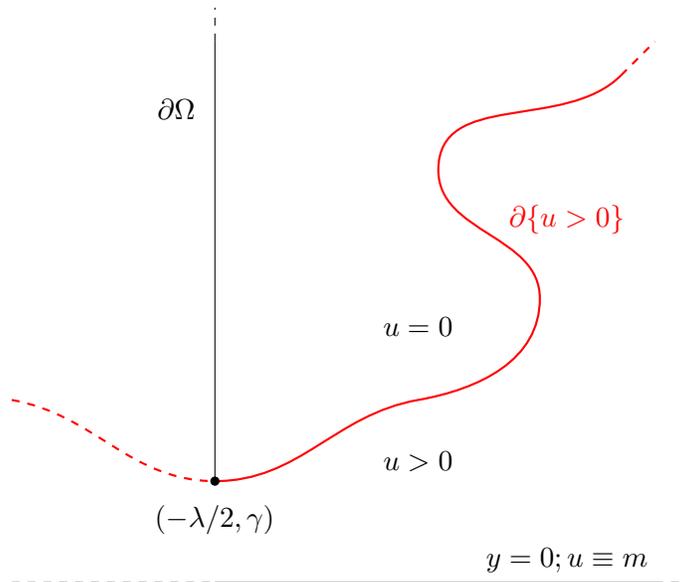

While we believe that \Cref{bgl} and \Cref{mainthm} are of interest in themselves, our original motivation is the theory of periodic traveling waves. Indeed, if one considers the steady irrotational flow over a flat impermeable bed of a two-dimensional inviscid incompressible fluid acted on by gravity, the equations of motion can be rewritten as a Bernoulli-type free boundary problem for a stream function (see, for example, \cite{MR2753609} and \cite{MR2604871}). By denoting $u$ the stream function, $\la$ the wavelength, and letting $m,h$ be renormalized constants related to mass flux and hydraulic head, respectively, one is lead to consider the free boundary problem 

\begin{equation}
\label{fbWW}
\left\{
\arraycolsep=1.4pt\def\arraystretch{1.6}
\begin{array}{rll}
\Delta u = & 0 & \text{ in } \Omega \cap \{u > 0\}, \\
|\nabla u| = & \sqrt{(h - y)_+} & \text{ on } \Omega \cap \partial\{u > 0\}, \\
u = & m & \text{ on } \{y = 0\}.
\end{array}
\right.
\end{equation}

Although many results on water waves have been obtained by mapping the domain of the fluid into a fixed domain in the complex plane by means of a hodograph transform (see, for example, \cite{MR869412}, \cite{MR666110}, \cite{MR1446239}, \cite{MR1883094}, \cite{MR513927}; see also \cite{MR2604871}, \cite{MR0502787}, \cite{MR869413}, \cite{MR2038344}), in recent years variational approaches have been proposed to tackle these kind of problems (see, for example, \cite{MR2915865}, \cite{GL18}, \cite{MR2810856}, \cite{MR2995099}, \cite{MR2928132}, \cite{MR2959383}). The advantages of considering a variational formulation of  problem (\ref{fbWW}) are twofold: it allows for more general geometries such as multiple air components, while at the same time it retains the physical intuition of the model. 

On the other hand, a free boundary approach for the existence of periodic water waves is a notoriously difficult problem as variational solutions to (\ref{fbWW}), i.e., minimizers of the functional $\J_h$ in the class 
\[
\K \coloneqq \left\{u \in H^1_{\operatorname{per}}(\Omega) : u(\cdot,0) = m\right\},
\]
are \emph{one-dimensional solutions of the form $u = u(y)$}, so that the free surface is \emph{flat}. 

For this reason, in \cite{GL18} we added as a constraint a vertical Dirichlet condition as in $\K_{\ga}$ and proved that choosing $\ga$ opportunely has the effect of eliminating trivial solutions from the domain of $\J_h$. For the convenience of the reader, the precise statements of our results are reported below in \Cref{background}.  

The boundary regularity for local minimizers of $\J_h$ in $\K_{\ga}$ and their free boundaries away from the points $(\pm\la/2, \ga)$ is well understood as a consequence of the aforementioned results (see \Cref{summary}). Indeed, due to the periodic boundary conditions below the line $\{y = \ga\}$, if the free boundary $\partial \{u > 0\}$ of a local minimizer touches the fixed boundary strictly below this line, then the classical  interior regularity of \cite{MR618549} forces it to be regular across periods. On the other hand, if the free boundary touches the fixed boundary strictly above that line, then it must detach tangentially from the fixed boundary; this would cause a cut in the fluid domain and the formation of a cusp on the free surface. Consequently, this work settles an important issue that was left open in \cite{GL18}.

\begin{figure}[h]
\centering
\begin{tikzpicture}[blend group=normal, scale=1.1]

\draw [dashed] (0,0) -- (0.3,0);
\draw (0.3,0) -- (3,0);
\draw [dashed] (3,0) -- (3.3,0);
\draw (0.3, 1.2) -- (0.3, 3);
\draw [dashed] (0.3, 3) -- (0.3, 3.3);
\draw [fill] (0.3,1.2) circle [radius=0.04];
\draw [red, thick] (0.3,1.6) to [out=90, in=185] (2.8,2.9);
\draw [red, thick, dashed] (2.8,2.9) to (3.1,2.92);
\draw [fill, red] (0.3,1.6) circle [radius=0.04];

\draw [dashed] (4,0) -- (4.3,0);
\draw (4.3,0) -- (7,0);
\draw [dashed] (7,0) -- (7.3,0);
\draw (4.3, 1.2) -- (4.3, 3);
\draw [dashed] (4.3, 3) -- (4.3, 3.3);
\draw [fill] (4.3,1.2) circle [radius=0.04];
\draw [red, thick] (4.3,1.2) to [out=0, in=215] (6.8,2.5);
\draw [red, thick, dashed] (6.8,2.5) to (7.1,2.7);
\draw [fill, red] (4.3,1.2) circle [radius=0.04];

\draw [dashed] (8,0) -- (8.3,0);
\draw (8.3,0) -- (11,0);
\draw [dashed] (11,0) -- (11.3,0);
\draw (8.3, 1.2) -- (8.3, 3);
\draw [dashed] (8.3, 3) -- (8.3, 3.3);
\draw [fill] (8.3,1.2) circle [radius=0.04];
\draw [red, thick] (8.3,0.8) to [out=0, in=215] (10.8,2.3);
\draw [red, thick, dashed] (10.8,2.3) to (11.1,2.5);
\draw [fill, red] (8.3,0.8) circle [radius=0.04];

\node [above] at (2, 0) {$y = 0; u \equiv m$};
\node [above] at (6, 0) {$y = 0; u \equiv m$};
\node [above] at (10, 0) {$y = 0; u \equiv m$};
\end{tikzpicture}
\caption{Qualitative behavior of the free boundary near the fixed boundary.}
\label{summary}
\end{figure}
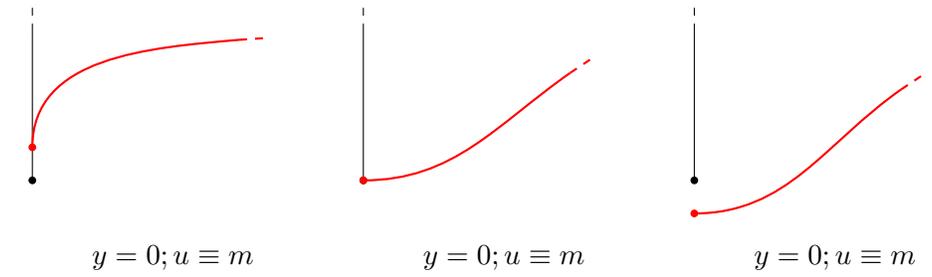

Let us also remark that if one were able to prove that for some choice of the parameters $m, h, \ga, \la$ there exists a local minimizer with the property that every contact point belongs to the set $\{y \le \ga\}$, then by \Cref{mainthm} such a minimizer would solve (\ref{fbWW}) in the entire half-plane $\RR^2_+$. This would provide the first variational proof of the existence of regular water waves which does not rely on Nekrasov's equation (see the classical papers \cite{MR138284}, \cite{MR0502787}). This is ongoing work.

Independently of their applications to the theory of water waves, we believe that the techniques presented in this paper are of interest in themselves and could be applied in other contexts.

\subsection{Plan of the paper}
\noindent Our paper is organized as follows: for the convenience of the reader, in \Cref{background} we recall some well-known results on the existence and regularity of local minimizers of the functional $\J_h$ and state the non-flatness result of \cite{GL18} (see \Cref{exofgamma}). In Section 3 we recall basic properties of symmetric local minimizers and prove that their free boundaries can be described by continuous $y$-graphs (see \Cref{gcont}). Section 4 is dedicated to the proof of \Cref{bgl}. In Section 5 we collect some preliminary results on blow-up limits which will prove useful in the following sections. In Section 6 we extend Weiss' boundary monotonicity formula to our framework. Finally, in Section 7 we give a complete characterization of blow-up solutions (see \Cref{class}) and conclude this work by presenting the proof of  \Cref{mainthm}.

\section{Background results}
\label{background}

The following theorem summarizes the classical existence and regularity theory due to Alt and Caffarelli \cite{MR618549}.
\begin{thm}
\label{exofmin}
Given $m, h, \ga, \la > 0$, let $\Omega$, $\K_{\ga}$, and $\J_h$ be defined as in $(\ref{strip})$, $(\ref{Jper})$, and $(\ref{Kg})$, respectively. Then the minimization problem for $\J_h$ in $\K_{\ga}$ admits a solution. Furthermore, if $u \in \K_{\ga}$ is a local minimizer of the functional $\J_h$, the following hold: 
\begin{itemize}
\item[$(i)$] $u$ is subharmonic in $\Omega$;
\item[$(ii)$] $u$ is locally Lipschitz continuous in $\Omega$; 
\item[$(iii)$] $u$ is harmonic in the set $\{u > 0\}$;
\item[$(iv)$] for any subset $K$ compactly contained in $(-\la/2,\la/2) \times (0,h)$ and any $\bm{x} \in \partial\{u > 0\} \cap K$ there exist a number $\rho > 0$, an analytic function $f$, and a set of local coordinates such that the free boundary $\partial \{u > 0\}$ coincides with the graph of $f$ in $B_{\rho}(\bm{x})$. Furthermore, if we let $\nu$ be the inner unit normal vector to $\{u > 0\}$ at $\bm{x} = (x,y)$, then
\[
\partial_{\nu}u(\bm{x}) = \sqrt{h - y}.
\] 
\end{itemize}
\end{thm}

\begin{proof}
Since $u_0$ defined in (\ref{u0}) belongs to $\K_{\ga}$ and is such that $\J_h(u_0) < \infty$, the solvability of the minimization problem for $\J_h$ in $\K_{\ga}$ follows from Theorem 1.3 in \cite{MR618549} (see also Theorem 2.2 in \cite{MR2915865}). The proofs for statements $(i)$ through $(iv)$ can also be found in \cite{MR618549}; more precisely, we refer the reader to Lemma 2.2, Corollary 3.3, Lemma 2.4, and Theorem 8.4.
\end{proof}

\begin{rmk}
\label{representative}
In view of property $(i)$, throughout the rest of the paper we work with the precise representative 
\[
u(\bm{x}) = \lim_{r \to 0^+}\fint_{B_r(\bm{x})}u(\bm{y})\,d\bm{y}, \quad \bm{x} \in \Omega.
\]
\end{rmk}

The next result states that for opportune choices of the parameter $\ga$, the minimization problem for $\J_h$ in $\K_\ga$ admits nontrivial solutions, i.e., minimizers which are not of the form $u = u(y)$ and whose free boundaries are not flat.
\begin{thm}[Theorem 1.1 and Corollary 4.3 in \cite{GL18}]
\label{exofgamma}
Given $m,h, \la> 0$, let $\Omega$, $\J_h$, and $\K_\ga$ be defined as in $(\ref{strip})$, $(\ref{Jper})$, and $(\ref{Kg})$, respectively. Let 
\[
h^\# \coloneqq 3\left(\frac{m}{2}\right)^{2/3}, \quad h^* \coloneqq 3\left(\frac{m}{\sqrt{2}}\right)^{2/3},
\]
and, for $h > h^{\#}$, let $t_h$ be the first positive root of the cubic polynomial 
\[
t^3 - ht^2 + m^2 = 0.
\]
Furthermore, for $h \in (h^{\#},h^*)$, let $\tau_h > t_h$ be the unique value such that
\[
\frac{m^2}{t_h} + \frac{h^2 - (h - t_h)^2}{2} = \frac{m^2}{\tau_h} + \frac{h^2 - (h - \min\{h,\tau_h\})^2}{2},
\]
and $\tau_h = t_h = 2h/3$ if $h = h^{\#}$. Let 
\begin{equation}
\label{gammarange}
\left\{
\arraycolsep=1.4pt\def\arraystretch{1.6}
\begin{array}{rll}
\ga \in & (0,\infty) & \text{ if } h < h^{\#}, \\
\ga \in & (0,t_h) \cup (\tau_h,\infty) & \text{ if } h^{\#} \le h < h^*, \\
\ga \in & (0,t_h) & \text{ if } h \ge h^*.
\end{array}
\right.
\end{equation}
Then, for every global minimizer $u \in \K_{\ga}$ of the functional $\J_h$, the following hold:
\begin{itemize}
\item[$(i)$] $u$ is not of the form $u = u(y)$;
\item[$(ii)$] the free boundary $\partial \{u > 0\}$ is not flat, i.e., it does not coincide with a line of the form $\{y = k\}$, for some $k > 0$.
\end{itemize}
\end{thm}

\section{Symmetric minimizers and their free boundaries}
The existence of symmetric minimizers in the sense of \Cref{sm} was previously observed in Theorem 5.10 in \cite{MR2915865} (see also Theorem 1.6 and Remark 5.14 \cite{GL18}). In particular, we recall that for a symmetric minimizer $u \in \K_{\ga}$, the portion of the free boundary $\partial \{u > 0\}$ in $\{-\la/2 < x < 0\}$ can be described by the graph of a function $x = g(y)$, where $g \colon (0,h) \to [-\la/2,0]$ is defined via 
\begin{equation}
\label{gdef}
g(y) \coloneqq \inf \{x \in (-\la/2,0) : u(x,y) > 0\}.
\end{equation}

\begin{prop}
\label{gcont}
Given $m, h, \la > 0$ and $\ga$ as in $(\ref{gammarange})$, let $\Omega$, $\J_h$, and $\K_{\ga}$ be defined as in $(\ref{strip})$, $(\ref{Jper})$, and $(\ref{Kg})$, respectively. Let $u \in \K_{\ga}$ be a symmetric minimizer of $\J_h$ in the sense of \Cref{sm} and let $g$ be defined as above. Then $g$ is a continuous function.
\end{prop}

\begin{proof} We divide the proof into several steps.
\newline
\textbf{Step 1:} We begin by showing that if $\partial \{u > 0\}$ contains the line segment $S$ of endpoints $(\ell,k), (L,k)$, with $\ell < L$ and $k < h$, then $\partial \{u > 0\} = \{y = k\}$, i.e., the free boundary of $u$ coincides with a line segment in $\Omega$, a contradiction to \Cref{exofgamma} $(ii)$. Without loss of generality, we can assume that $S$ is maximal, i.e., for every line segment $S'$ such that $S \subset S'$ and $S' \subset \partial\{u > 0\}$, it must be that $S = S'$. If $\ell = -\la/2$ and $L = \la/2$ there is nothing to do. Then assume without loss that $L < \la/2$. Since $k < h$, by \Cref{exofmin} $(iv)$ we can find a number $\rho > 0$, an analytic function $f$, and a set of local coordinates such that the free boundary $\partial \{u > 0\}$ coincides with the graph of $f$ in $B_{\rho}((L,k))$ in the local coordinates. In turn, $f$ agrees with an affine function on a subinterval of its domain, and so by analyticity it must be equal to the same affine function on its whole domain; this contradicts the maximality of $S$. 
\newline
\textbf{Step 2:} Next, we show that both one-sided limits 
\[
\lim_{y \to \bar{y}^+}g(y) \quad \text{ and } \quad \lim_{y \to \bar{y}^-}g(y) 
\]
exist for every $\bar{y} \in (0,h)$. To see this, suppose that
\[
L \coloneqq \limsup_{y \to \bar{y}^+}g(y) > \liminf_{y \to \bar{y}^+}g(y) \eqqcolon \ell;
\] 
then we can find two sequences $\{y_n\}_n, \{z_n\}_n$ such that $y_n \searrow \bar{y}$, $z_n \in (y_{n + 1}, y_n)$ and
\[
\lim_{n \to \infty}g(y_n) = L, \qquad \lim_{n \to \infty}g(z_n) = \ell.
\]
Let $\bm{y} \coloneqq (L, \bar{y})$. We claim that $\bm{y} \in \partial \{u > 0\}$. To prove the claim, first observe that there exists a $\delta > 0$ such that $B_{\delta}(\bm{y}) \subset \Omega$, and notice that $u(\bm{y}) = 0$ since $u$ is continuous in $\Omega$ and by assumption $u(g(y_n),y_n) = 0$ for every $n \in \NN$. Given $\eta > 0$, if $n \in \NN$ is large enough then $(g(y_n), y_n) \in B_{\eta}(\bm{y})$ and since by assumption $(g(y_n), y_n) \in \partial \{u > 0\}$ then there exists $\bm{x}_n \in B_{\eta}((g(y_n), y_n))$ such that $u(\bm{x}_n) > 0$. This shows that $\bm{y} \in \partial \{u > 0\}$. Again by \Cref{exofmin}, there exists $\rho > 0$ such that $B_{\rho} \cap \partial \{u > 0\}$ is the graph of a smooth function (in an opportunely defined set of coordinates centered at the point $\bm{y}$). Denote by $\kappa$ the Lipschitz constant of this function in $(-\rho,\rho)$ and notice that the length of $\partial \{u > 0\}$ in $B_{\rho}(\bm{y})$ cannot exceed $2\rho\sqrt{1 + \kappa^2}$.  On the other hand, we observe that for $n$ large enough one also has that $(g(y_n), y_n) \in B_{\rho/2}(\bm{y})$ and $(g(z_n), z_n) \notin B_{\rho}(\bm{y})$, thus showing that the length of $\partial \{u > 0\}$ cannot be finite. We have therefore reached a contradiction. The proof in the other case is similar and therefore we omit the details.
\newline
\textbf{Step 3:} The previous step shows that $g$ cannot have essential discontinuities. To exclude jump discontinuities it is enough to notice that these would correspond to horizontal line segments in the free boundary of $u$, a behavior that is ruled out in the first step. Finally, in view of Corollary 3.6 in \cite{MR618549}, we see that removable discontinuities are also not possible. This concludes the proof.
\end{proof}

\section{Proof of \Cref{bgl}}
\begin{proof}[Proof of \Cref{bgl}]
It is enough to show that there exists a constant $C$ such that for every $\mu > 0$ sufficiently small (with respect to $\la, h, \ga$ and $h - \ga$) 
\begin{equation}
\label{bglwts}
|\nabla u(\bm{y})| \le C
\end{equation}
for $\bm{y} \in \Omega \cap B_{2\mu}(\bm{x}_0) \setminus B_{\mu}(\bm{x}_0)$. For $\bm{x} \in B_8(\bm{0})$ and $\mu$ small enough, let $w$ be the rescaled function
\[
w(\bm{x}) \coloneqq \frac{u(\bm{x}_0 + \mu \bm{x})}{\mu}.
\]
Then $w$ is harmonic in $\{w > 0\}$ and for $\bm{x} = (x,y) \in \partial \{w > 0\} \cap \Omega$, by \Cref{exofmin} $(iv)$ we have that
\begin{equation}
\label{bernw}
\partial_{\nu}w(\bm{x}) = \partial_{\nu}u(\bm{x}_0 + \mu\bm{x}) = \sqrt{h - \ga - \mu y},
\end{equation}
where $\nu$ is the interior unit normal vector to $\{w > 0\}$ at $\bm{x}$. Clearly, to prove (\ref{bglwts}) is enough to show that
\begin{equation}
\label{bglwts2}
|\nabla w(\bm{x})| \le C
\end{equation}
for $\bm{x} \in \{w > 0\} \cap B_2^+(\bm{0}) \setminus B_1^+(\bm{0})$, where $B_r^+(\bm{0}) \coloneqq B_r(\bm{0}) \cap \{x > 0\}$. For $\bm{x} \in B_8(\bm{0})$ we define
\[
d(\bm{x}) \coloneqq \dist(\bm{x}, \partial \{w > 0\}), \quad D(\bm{x}) \coloneqq \dist(\bm{x}, \{(0,y) : y \ge 0\}).
\]
The proof of (\ref{bglwts2}) is divided into several steps. 
\newline
\textbf{Step 1:} In this first step we show that in order to obtain (\ref{bglwts2}), it is enough to prove that for every $\bm{x} \in B_2^+(\bm{0}) \setminus B_1^+(\bm{0})$ either 
\begin{equation}
\label{interiorcase}
w(\bm{x}) \le c\min\{d(\bm{x}), D(\bm{x})\} 
\end{equation}
or there exists $\rho > 0$ such that $\bm{x} = (x,y) \in B_{\rho/2}^+(0,y)$ and for every $\bm{y} \in B_{\rho}^+(0,y)$
\begin{equation}
\label{boundarycase}
w(\bm{y}) \le c \rho.
\end{equation}
Indeed, assume that $\bm{x} \in B_2^+(\bm{0}) \setminus B_1^+(\bm{0})$ is such that (\ref{interiorcase}) is satisfied. Then, if let $\delta(\bm{x}) \coloneqq \min\{d(\bm{x}), D(\bm{x})\}$, we have that $w$ is harmonic in $B_{\delta(\bm{x})}(\bm{x})$ and 
\[
|\nabla w(\bm{x})| \le \sup\left\{|\nabla w(\bm{y})| : \bm{y} \in B_{\delta(\bm{x})/2}(\bm{x})\right\} \le \frac{4}{\delta(\bm{x})}\sup\left\{w(\bm{y}) : \bm{y} \in B_{\delta(\bm{x})}(\bm{x})\right\} \le 4c,
\]
where the second inequality follows from standard interior gradient estimates (see Theorem 2.10 in \cite{gilbargtrudinger}). Similarly, for every $\bm{x} = (x,y) \in B_2^+(\bm{0}) \setminus B_1^+(\bm{0})$ such that (\ref{boundarycase}) holds, we see that 
\[
|\nabla w(\bm{x})| \le \sup\left\{|\nabla w(\bm{y})| : \bm{y} \in B_{\rho/2}^+(0,y)\right\} \le \frac{K}{\rho}\sup\left\{w(\bm{y}) : \bm{y} \in  B_{\rho}^+(0,y)\right\} \le Kc,
\]
where in the second inequality we have used Theorem 4.11 in \cite{gilbargtrudinger}.
\newline
\textbf{Step 2:} Let $c_0 > 3 \sqrt{h} \log 2 $. We claim that for every $\bm{x} \in B_4^+(\bm{0})$ for which $d(\bm{x}) < D(\bm{x})$ then
\begin{equation}
\label{unifbound}
w(\bm{x}) \le c_0 d(\bm{x}).
\end{equation}
Notice that if $w(\bm{x}) = 0$ then there is nothing to do, therefore, we assume without loss of generality that $w(\bm{x}) > 0$. Since $B_{d(\bm{x})}(\bm{x}) \subset \{w > 0\}$ we have that $w$ is harmonic in $B_{d(\bm{x})}(\bm{x})$ and by definition there must be $\bar{\bm{x}} \in \partial B_{d(\bm{x})}(\bm{x}) \cap \partial \{w > 0\}$. Suppose that 
\[
w(\bm{x}) > c_0 d(\bm{x}).
\]
Then, by Harnack's inequality (see Exercise 2.6 in \cite{gilbargtrudinger}), 
\[
w(\bm{y}) \ge \frac{w(\bm{x})}{3} > \frac{c_0d(\bm{x})}{3}
\]
for every $\bm{y} \in B_{d(\bm{x})/2}(\bm{x})$. Let $v$ be the harmonic function in the annulus $B_{d(\bm{x})}(\bm{x}) \setminus B_{d(\bm{x})/2}(\bm{x})$ which satisfies the boundary conditions
\[
\left\{
\arraycolsep=1.4pt\def\arraystretch{1.6}
\begin{array}{ll}
\displaystyle v = \frac{c_0d(\bm{x})}{3} & \text{ on } \partial B_{d(\bm{x})/2}(\bm{x}), \\
\displaystyle v = 0 & \text{ on } \partial B_{d(\bm{x})}(\bm{x}).
\end{array}
\right.
\]
Writing $v$ in polar coordinates centered at $\bm{x}$, $v$ must be the radial function 
\[
r \mapsto \frac{c_0d(\bm{x})}{3 \log 2} \log \left(\frac{d(\bm{x})}{r}\right)
\]
By the maximum principle for harmonic functions, $v \le w$ in the annulus, and since equality holds at the point $\bar{\bm{x}}$, it follows that
\[
\frac{c_0}{3\log 2} = \partial_{\nu}v(\bar{\bm{x}}) \le \partial_{\nu}w(\bar{\bm{x}}) \le \sqrt{h},
\]
where in the last inequality we have used (\ref{bernw}). In turn,
\[
c_0 \le 3 \sqrt{h} \log 2 , 
\]
which is in contradiction with our choice of $c_0$.
\newline
\textbf{Step 3: } Let 
\[
U_0 \coloneqq \left\{\bm{x} = (x,0) : 1 < x < 4 \right\}.
\]
In this step we show that there exists a constant $c_1 \ge c_0$, independent of $\mu$, such that 
\begin{equation}
\label{unifD}
w(\bm{x}) \le c_1D(\bm{x})
\end{equation}
for every $\bm{x} \in U_0 \cap \{w > 0\}$ with $1 \le D(\bm{x}) \le d(\bm{x})$. In view of (\ref{x0acc}), we can find a point $\bm{z}_0 = (s_0,t_0)$ on $\partial \{w > 0\} \cap B_{1/4}^+(\bm{0})$. Then, for every $s$ such that $(s,t_0) \in B_4^+(\bm{0}) \setminus B_1^+(\bm{0})$, we must have that $d(s,t_0) < D(s,t_0)$. Consequently, 
\[
w(s,t_0) \le c_0d(s,t_0),
\]
where $c_0$ is the constant given in the previous step. Notice that by assumption $d(\bm{x}) \ge D(\bm{x}) \ge 1$, and therefore $B_{1/2}(\bm{x}) \subset \{w > 0\}$. Moreover, the ball $B_{1/4}(\bm{x})$ contains the point $(x,t_0)$ and Harnack's inequality then yields
\begin{equation}
\label{w3w}
w(\bm{x}) \le 3w(x,t_0).
\end{equation}
On the other hand,
\begin{equation}
\label{wt0}
w(x,t_0) \le c_0d(x,t_0) < c_0D(x,t_0) \le c_0\sqrt{x^2 + t_0^2} \le \frac{\sqrt{17}c_0x}{4} = \frac{\sqrt{17}c_0D(\bm{x})}{4},
\end{equation}
where in the last inequality we have used the fact that $|t_0| \le 1/4 \le x/4$. The desired inequality (\ref{unifD}) follows directly from (\ref{w3w}) and (\ref{wt0}).
\newline
\textbf{Step 4:}
The purpose of this step is to show that (\ref{unifD}) holds, possibly with a larger constant, at every point $\bm{x} \in B_4^+(\bm{0}) \setminus B_1^+(\bm{0})$, such that $y < 0$ and $D(\bm{x}) \le d(\bm{x})$. We begin by considering the case 
\[
\bm{x} \in U_1 \coloneqq \left\{\bm{z} = (s,t) \in B_4^+(\bm{0}) \setminus B_1^+(\bm{0}) : t < 0 \text{ and } \dist(\bm{z}, U_0) < \frac{1}{4}\right\}.
\]
Reasoning as in the previous step, we see that since $d(\bm{x}) \ge D(\bm{x}) \ge 1$ we are in a position to apply Harnack's inequality in $B_{1/4}(\bm{x}) \subset B_{1/2}(\bm{x}) \subset \{w > 0\}$ to conclude that
\[
w(\bm{x}) \le 3w(\bm{z}_1)
\]
for every $\bm{z}_1 \in U_0$ such that $|\bm{x} - \bm{z}_1| < 1/4$. Additionally, it follows from steps two and three that
\[
w(\bm{z}_1) \le c_1 \min\{d(\bm{z}_1),D(\bm{z}_1)\} \le c_1D(\bm{z}_1) \le 4c_1 \le 4c_1D(\bm{x}).
\]
Define the sets $U_i$, $i \ge 2$, recursively via
\[
U_i \coloneqq \left\{\bm{z} = (s,t) \in \left (B_4^+(\bm{0}) \setminus B_1^+(\bm{0}) \right) \setminus {\textstyle\bigcup_{j = 1}^{i - 1}}U_j : t < 0 \text{ and } \dist(\bm{z}, U_{i - 1}) < \frac{1}{4}\right\},
\]
and notice that by simple geometric considerations
\[
\left(B_4^+(\bm{0}) \setminus B_1^+(\bm{0})\right) \cap \{t \le 0\} = \bigcup_{i = 0}^{16}U_i.
\]
In particular, if $\bm{x} \in U_i$ is such that $D(\bm{x}) \le d(\bm{x})$ then an iteration of the argument above yields
\[
w(\bm{x}) \le 12^ic_1D(\bm{x}).
\]
\textbf{Step 5:} We are left to consider the case where $\bm{x} = (x,y) \in B_2^+(\bm{0}) \setminus B_1^+(\bm{0})$ is such that $y > 0$, $\bm{x} \in \{w > 0\}$, and $D(\bm{x}) \le d(\bm{x})$. Suppose that there exists a sequence $\{x_n\}_{n \in \NN} \subset B_4^+(\bm{0}) \setminus B_1^+(\bm{0})$ such that $\bm{x}_n \to \bm{x}$ and such that $d(\bm{x}_n) < D(\bm{x}_n)$ for every $n$. Then necessarily $d(\bm{x}) = D(\bm{x})$ and by (\ref{unifbound})
\begin{equation}
\label{wlim}
w(\bm{x}) = \lim_{n \to \infty}w(\bm{x}_n) \le \lim_{n \to \infty}cd(\bm{x}_n) = cD(\bm{x}).
\end{equation}
Hence, we can assume that such a sequence does not exist. Then there is $0 < \delta < y$ such that for every $t \in (y - \delta, y + \delta)$ the point $(x,t)$ is such that $D(x,t) \le d(x,t)$, and in particular $w(s,t) > 0$ for every $0 < s < x$. We define
\begin{align*}
a &\ \coloneqq \inf \left\{t \le y : \text{ for every } t < \bar{t} < y + \delta, w(s,\bar{t}) > 0 \text{ for every } s \text{ small} \right\},\\
b &\ \coloneqq \sup \left\{t \ge y : \text{ for every } y - \delta < \bar{t} < t, w(s,\bar{t}) > 0 \text{ for every } s \text{ small}\right\}.
\end{align*}
Notice that by (\ref{x0acc}), $a \ge 0$. Moreover, $y \in (a,b)$, and it follows from the definition that if $b < \infty$, every point of the form $(s,a)$ and $(s,b)$, $s > 0$, is the limit of a sequence of points $\{\bm{x}_n\}_{n \in \NN}$ with the property that $d(\bm{x}_n) < D(\bm{x}_n)$. In turn, (\ref{unifbound}) and (\ref{wlim}) imply that 
\begin{equation}
\label{wsa}
w(s,a) \le cs, \quad w(s,b) \le cs,
\end{equation}
for every $s > 0$ such that the points $(s,a), (s,b) \in B_4^+(\bm{0})$. Assume first that $y - a < b - y$ and fix $\e > 0$ small enough so that 
\begin{equation}
\label{pi4-e}
1 - \tan \theta \le \frac{1}{4}, \quad \theta \coloneqq \frac{\pi}{4} - \e.
\end{equation}
\emph{Case 1}: Assume that $y - a \le x \tan \theta$. Let $\bar{\bm{x}} = (x,a)$ and notice that 
\[
|\bm{x} - \bar{\bm{x}}| = y - a \le x\tan \theta < x = D(\bm{x}).
\]
Since by assumption $D(\bm{x}) \le d(\bm{x})$ we have that $B_{x \tan \theta}(\bm{x}) \subset B_{D(\bm{x})}(\bm{x}) \subset \{w > 0\}$ and by Harnack's inequality we can find a constant $c = c(\e)$ such that
\[
w(\bm{x}) \le cw(\bar{\bm{x}}) \le cx = cD(\bm{x}),
\]
where in the last inequality we have used (\ref{wsa}).
\newline
\emph{Case 2}: Assume that $x \tan \theta < y - a \le x$ and let $\widehat{\bm{x}} = (x, a + x \tan \theta)$. By (\ref{pi4-e}) we see that 
\[
|\bm{x} - \widehat{\bm{x}}| \le x\left(1 - \tan \theta \right) \le \frac{x}{4}.
\]
In turn, $B_{x/2}(\bm{x}) \subset \{w > 0\}$, and similarly to above, by Harnack's inequality,
\begin{equation}
\label{y=a+x}
w(\bm{x}) \le 3w(\widehat{\bm{x}}) \le cD(\bm{x}),
\end{equation}
where in the last inequality follows from the fact that $\widehat{\bm{x}}$ satisfies the conditions of Case 1.
\newline
\emph{Case 3:} Assume that
\[
\frac{3}{4}(y - a) \le x < y - a.
\]
Since $y - a < b - y$ it follows that $B_{\frac{1}{2}(y - a)}(\bm{x}) \subset \{w > 0\}$, and therefore
\[
w(\bm{x}) \le 3w(y - a, y) \le c(y - a) \le \frac{4}{3}cx = \frac{4}{3}cD(\bm{x}),
\]
where in the second inequality we have used the fact that the point $(y - a, y)$ satisfies the conditions of Case 2.
\newline
\emph{Case 4}: Assume that
\[
\frac{1}{2}(y - a) \le x < \frac{3}{4}(y - a).
\]
Then $(3(y - a)/4, y)$ satisfies of the conditions of Case 3, and so, reasoning as above, we obtain that
\[
w(\bm{x}) \le cw(3(y - a)/4, y-a) \le c(y - a) \le 2cx = 2cD(\bm{x}).
\]
\emph{Case 5:} Finally, assume that $x < (y - a)/2$. Notice that $B_{y - a}^+(0,y) \subset \{w > 0\}$ by the non-decreasing property of symmetric minimizers. Then, for every $\bm{y} \in B_{(y - a)/2}^+(0,y)$, by the boundary Harnack principle (see Theorem 11.5 in \cite{MR2145284}) we have that
\[
w(\bm{y}) \le Mw(y - a, y) \le Mc(y - a),
\]
where in the last inequality we used (\ref{y=a+x}). If $y - a > b - y$ then $4 \ge 2y - a > b$ and therefore we can repeat the same argument as above. This concludes the proof.
\end{proof}

\section{Blow-up limits}
\label{busec}
Given a local minimizer $u \in \K_{\ga}$ of the functional $\J_h$, consider a sequence $\rho_n \to 0^+$, a real number $R > 0$, and for every $n \in \NN$ sufficiently large define the rescaled functions
\begin{equation}
\label{blowup}
u_{n}(\bm{z}) \coloneqq \frac{u(\bm{x}_0 + \rho_n \bm{z})}{\rho_n},
\end{equation}
where $\bm{z} \in B_R(\bm{0})$, and $\bm{x}_0 = (-\la/2,\ga)$. Notice that if $\nabla u$ is bounded in a neighborhood of $\bm{x}_0$ (a condition that is guaranteed under the assumptions of \Cref{bgl}), then for every $n$ large enough
\[
|\nabla u_n(\bm{z})| = |\nabla u(\bm{x}_0 + \rho_n \bm{z})| \le C,
\]
where $C$ is a positive constant independent of $n$ and $\bm{z}$. Since $u_n(\bm{0}) = 0$ for every $n \in \NN$, it follows that there exist a subsequence (which we don't relabel) and a function $w \in W^{1,\infty}_{\loc}(\RR^2)$ such that for every $R > 0$,
\begin{equation}
\begin{aligned}
\label{blowupconv}
u_n \to &\, w \quad \ \ \, \text{ in } C^{0,\alpha}(B_R(\bm{0})) \text{ for all } 0 < \alpha < 1, \\
\nabla u_n \overset{\ast}\rightharpoonup &\, \nabla w \quad \text{ in } L^{\infty}(B_R(\bm{0});\RR^2). 
\end{aligned}
\end{equation}
The function $w$ is called a \emph{blow-up limit} of $u$ at $\bm{x}_0$ with respect to the sequence $\{\rho_n\}_n$. 

\subsection{Non-degeneracy properties of blow-up limits}
\begin{prop}
\label{competition2}
Given $m, h, \ga, \la > 0$ and $k \in (0,1)$, there exists a positive constant $C_{\min}(k)$ such that for every (local) minimizer $u$ of $\J_h$ in $\K_{\ga}$ and for every (small) ball $B_r(\bm{x}) \subset \Omega$, $\bm{x} = (x,y)$, if 
\[
\frac{1}{r} \fint_{\partial B_r(\bm{x})}u\, d\mathcal{H}^{1} \le C_{\min}(k)\sqrt{(h - y - kr)_+},
\]
then $u \equiv 0$ in $B_{kr}(\bm{x})$. Moreover, if $0 < r < \la$, the result is still valid for balls not contained in $\Omega$, provided $B_r(\bm{x}) \cap \{(\pm \la/2, \ga)\} = \emptyset$.
\end{prop}

For a proof of \Cref{competition2} we refer to Lemma 3.4 and Remark 3.5 in \cite{MR618549}; see also Theorem 3.6 and Remark 5.2 in \cite{MR2915865}. 

\begin{lem}
\label{nondeg}
Given $m, h, \la > 0$ and $\ga < h$, let $u \in \K_{\ga}$ be a local minimizer of $\J_h$ and let $w$ be a blow-up limit of $u$ at $\bm{x}_0 =  (-\la/2,\ga)$ with respect to the sequence $\{\rho_n\}_n$. Furthermore, assume that there exist a constant $\beta \ge 1$ and a sequence of points $\bm{x}_n \in \partial \{u > 0\} \cap \Omega$ such that 
\begin{equation}
\label{rbr}
\rho_n \le |\bm{x}_n - \bm{x}_0| \le \beta \rho_n
\end{equation}
for every $n$ large enough. Then $w$ is not identically equal to zero.
\end{lem}
\begin{proof}
By assumption, there exists a sequence of radii $\{r_n\}_n$, $1 \le r_n \le \beta$, such that 
\[
\partial B_{\rho_nr_n}(\bm{x}_0) \cap \partial \{u > 0\} \cap \Omega \neq \emptyset.
\]
Thus, for every $n \in \NN$ sufficiently large,
\[
\bm{z}_n \coloneqq \frac{\bm{x}_n - \bm{x}_0}{\rho_n} \in \partial B_{r_n}(\bm{0}) \cap \partial \{u_n > 0\} \cap \{s > 0\},
\]
and furthermore we can assume that  $h - \ga - 2\beta \rho_n > 0$. Given $k \in (0,1)$, for every such $n$, consider the ball $B_{r_n}(\bm{z}_n)$ and observe that by the change of variables $\bm{x} = \bm{x}_0 + \rho_n\bm{z}$, (\ref{blowup}), and \Cref{competition2} (provided that $\rho_nr_n$ is sufficiently small) 
\begin{align*}
\frac{1}{r_n} \fint_{\partial B_{r_n}(\bm{z}_n)}u_n\, d\mathcal{H}^{1} = &\ \frac{1}{2\pi  \rho_n r_n^2}\int_{\partial B_{r_n}(\bm{z}_n)}u(\bm{x}_0 + \rho_n \bm{z})\,d\mathcal{H}^1(\bm{z}) \\
= &\ \frac{1}{2 \pi \rho_n^2r_n^2}\int_{\partial B_{\rho_nr_n}(\bm{x}_n)}u(\bm{x})\,d\mathcal{H}^1(\bm{x}) \\
= &\ \frac{1}{\rho_nr_n}\fint_{\partial B_{\rho_nr_n}(\bm{x}_n)}u\,d\mathcal{H}^1 \ge C_{\min}(k)\sqrt{(h - y_n - k\rho_nr_n)_+}.
\end{align*}
In addition, we notice that by (\ref{rbr}), $y_n \le \ga + \rho_nr_n$, and therefore
\begin{equation}
\label{unw>0}
\frac{1}{r_n} \fint_{\partial B_{r_n}(\bm{z}_n)}u_n\, d\mathcal{H}^{1} \ge C_{\min}(k)\sqrt{h - \ga - 2\beta \rho_n}. 
\end{equation}
Let $\bar{\bm{z}}_n$ be such that $u_n(\bar{\bm{z}}_n) = \sup\{u_n(\bm{z}) : \bm{z} \in \partial B_{r_n}(\bm{z}_n)\}$. Then, by (\ref{unw>0}) we see that 
\begin{equation}
\label{unw>01}
u_n(\bar{\bm{z}}_n) \ge \fint_{\partial B_{r_n}(\bm{z}_n)}u_n\, d\mathcal{H}^{1} \ge r_nC_{\min}(k)\sqrt{h - \ga - 2\beta \rho_n}.
\end{equation}
Eventually extracting a subsequence (which we don't relabel), we can find a point $\bar{\bm{z}}$ such that $\bar{\bm{z}}_n \to \bar{\bm{z}}$. Consequently, by the uniform convergence of $u_n$ to $w$, (\ref{unw>01}), and the fact that $r_n \ge 1$ for every $n$, we obtain
\[
w(\bar{\bm{z}}) = \lim_{n \to \infty}u_n(\bar{\bm{z}}_n) \ge \lim_{n \to \infty}r_nC_{\min}(k)\sqrt{h - \ga - 2\beta \rho_n} \ge C_{\min}(k)\sqrt{h - \ga} > 0.
\]
This concludes the proof.
\end{proof}

\subsection{Additional properties of blow-up limits}
The following classical lemma, due to Alt and Caffarelli, is a consequence of \Cref{competition2}; for a proof we refer to Section 4.7 in \cite{MR618549}.

\begin{lem}
\label{fbchi}
Given $m, h, \la > 0$ and $\ga < h$, let $u \in \K_{\ga}$ be a local minimizer of $\J_h$ and let $w$ be a blow-up limit of $u$ at $\bm{x}_0$. Then, if $u_n$ is defined as in $(\ref{blowup})$,
\begin{itemize}
\item[$(i)$] $\partial \{u_n > 0\} \to \partial \{w > 0\}$ locally in Hausdorff distance in $\RR^2 \setminus \{(0,y) : y \ge 0\}$,
\item[$(ii)$] $\chi_{\{u_n > 0\}} \to \chi_{\{w > 0\}}$ in $L^1_{\loc}(\RR^2 \setminus \{(0,y) : y \ge 0\})$.
\end{itemize}
\end{lem}

\begin{thm}
\label{bumin}
Given $m, h, \la > 0$ and $\ga < h$, let $u \in \K_{\ga}$ be a local minimizer of $\J_h$ and let $w$ be a blow-up limit of $u$ at $\bm{x}_0 = (-\la/2,\ga)$. Then, for every $R > 0$, $w$ is a global minimizer of 
\begin{equation}
\label{Jbu}
\F_h(v) \coloneqq \int_{B_R(\bm{0})}\left(|\nabla v(\bm{z})|^2 + \chi_{\{v > 0\}}(\bm{z})(h - \ga)\right)\,d\bm{z}, 
\end{equation}
over the set
\begin{equation}
\begin{aligned}
\label{Kw}
\K(w,R) \coloneqq \left\{v \in H^1_{\loc}(\RR^2) : v = w \text{ on } \partial B_R(\bm{0}) \text { and } v(0,y) = 0 \text{ for } 0 < y < R \right\}.
\end{aligned}
\end{equation}
\end{thm}
The following proof is adapted from Lemma 5.4 in \cite{MR618549}.
\begin{proof}
For $u_n$ defined as in (\ref{blowup}) and $n$ large enough so that $0 < \ga - R\rho_n < \ga + R\rho_n < h$, let $\eta \in C_0^1(B_R(\bm{0});[0,1])$ and for $v \in \K(w,R)$ set
\[
v_n(\bm{z}) \coloneqq v(\bm{z}) + (1 - \eta(\bm{z}))(u_n(\bm{z}) - w(\bm{z}))
\]
and, for $\bm{x} \in B_{R\rho_n}(\bm{x}_0)$, define
\[
w_n(\bm{x}) \coloneqq \rho_nv_n\left(\frac{\bm{x} - \bm{x}_0}{\rho_n}\right).
\]
Notice that by (\ref{blowup}), $w_n = u$ on $\partial B_{R\rho_n}(\bm{x}_0)$ in the sense of traces and furthermore that $w(-\la/2,y) = 0$ for $\mathcal{L}^1$-a.e.\@ $y \in (\ga, \ga + R\rho_n)$. Then the minimality of $u$ implies that  
\[
\int_{B_{R\rho_n}(\bm{x}_0)}\left(|\nabla u|^2 + \chi_{\{u > 0\}}(h - y)\right)\,d\bm{x} \le \int_{B_{R\rho_n}(\bm{x}_0)}\left(|\nabla w_n|^2 + \chi_{\{w_n > 0\}}(h - y)\right)\,d\bm{x},
\]
and the change of variables $\bm{x} = \bm{x}_0 + \rho_n\bm{z}$, $\bm{z} = (s,t)$, then yields
\begin{equation}
\label{wtsbumin}
\int_{B_R(\bm{0})}\left(|\nabla u_n|^2 + \chi_{\{u_n > 0\}}(h - \ga - \rho_nt)\right)\,d\bm{z} \le \int_{B_R(\bm{0})}\left(|\nabla v_n|^2 + \chi_{\{v_n > 0\}}(h - \ga - \rho_nt)\right)\,d\bm{z}.
\end{equation}
Since 
\[
\nabla v_n(\bm{z}) = \nabla v(\bm{z}) + (1 - \eta(\bm{z}))(\nabla u_n(\bm{z}) - \nabla w(\bm{z})) - \nabla \eta(\bm{z})(u_n(\bm{z}) - w(\bm{z})), 
\]
we observe that 
\begin{align}
\label{gradx666}
|\nabla v_n|^2 - |\nabla u_n|^2 = &\ |\nabla v|^2 + |\nabla \eta|^2|u_n - w|^2 - 2(u_n - w)\nabla \eta \cdot \nabla v + (1 - \eta)^2|\nabla u_n - \nabla w|^2 \notag \\
&\quad \ \ \quad + 2(1 - \eta)(\nabla u_n - \nabla w)\cdot(\nabla v - \nabla \eta(u_n - w)) - |\nabla u_n|^2 \notag \\
\le &\ |\nabla v|^2 + |\nabla \eta|^2|u_n - w|^2 - 2(u_n - w)\nabla \eta \cdot \nabla v -2\nabla u_n \cdot \nabla w + |\nabla w|^2 \notag \\
&\quad \ \ \quad + 2(1 - \eta)(\nabla u_n - \nabla w)\cdot(\nabla v - \nabla \eta(u_n - w)),
\end{align}
where in the last inequality we have used the fact that $(1 - \eta)^2 \le 1$.
Fix $\e > 0$ and let $R_{\e} \coloneqq \{\bm{z} : \dist(\bm{z}, \{(0,y) : y \ge 0\}) < \e\}$. Then, by \Cref{fbchi}, it follows that
\begin{equation}
\label{Reloc}
\int_{B_R(\bm{0}) \setminus R_{\e}}\chi_{\{u_n > 0\}}(h - \ga - \rho_nt)\,d\bm{z} \to \int_{B_R(\bm{0}) \setminus R_{\e}}\chi_{\{w > 0\}}(h - \ga)\,d\bm{z}
\end{equation}
Using the fact that 
\[
\chi_{\{v_n > 0\}} \le \chi_{\{v > 0\}} + \chi_{\{\eta < 1\}},
\]
combining (\ref{wtsbumin}), (\ref{gradx666}), (\ref{Reloc}), letting $n \to \infty$, and using the fact that $u_n \rightharpoonup u$ in $H^1$, we deduce that 
\[
\int_{B_R(\bm{0})}|\nabla w|^2\,d\bm{z} + \int_{B_R(\bm{0}) \setminus R_{\e}}\chi_{\{w > 0\}}(h - \ga)\,d\bm{z} \le \int_{B_R(\bm{0})}\left(|\nabla v|^2 + (\chi_{\{v > 0\}} + \chi_{\{\eta < 1\}})(h - \ga)\right)\,d\bm{z}.
\]
Letting $\e \to 0^+$, by the monotone convergence theorem we see that
\[
\F_h(w) \le \int_{B_R(\bm{0})}\left(|\nabla v|^2 + (\chi_{\{v > 0\}} + \chi_{\{\eta < 1\}})(h - \ga)\right)\,d\bm{z}.
\]
The desired result follows from an application of the dominated convergence theorem, choosing a sequence of functions $\eta_k$ such that $\eta_k \nearrow 1$.
\end{proof}

The next result is commonly referred to as a non-oscillation lemma (see, for example, Lemma 6.1 in \cite{MR733897}, Lemma 5.2 in Chapter 3 of \cite{MR1009785}, and Lemma 2.4 in \cite{MR1291575}).
\begin{lem}
\label{nol}
Given $m, h, \la > 0$ and $\ga < h$, let $u \in \K_{\ga}$ be a local minimizer of $\J_h$ and let $w$ be a blow-up limit of $u$ at $\bm{x}_0 = (-\la/2,\ga)$. Assume that there exists an open set $G$ contained in $\{w > 0\}$, which is compactly supported in $\RR^2 \setminus \{(0,y) : y \ge 0 \}$ and bounded by the line segments
\[
\ell_i \coloneqq \{(s_i, t) : t_i < t < t_i + \e_i\}, \quad i = 1,2,
\]
and two non intersecting arcs $\phi_i$, $i = 1,2$, contained in the free boundary $\partial \{w > 0\}$ and joining the points $(s_1,t_1)$ with $(s_2,t_2)$ and $(s_1, t_1 + \e_1)$ with $(s_2, t_2 + \e_2)$. Then 
\[
|s_2 - s_1| \le \frac{\sup_G|\nabla w|(\e_1 + \e_2)}{2\sqrt{h - \ga}}.
\]
\end{lem}
\begin{proof}
Observe that $w$ is harmonic in $G$ and therefore by the divergence theorem
\[
0 = \int_{\partial G}\partial_{\nu}w\,d\mathcal{H}^1 = \sum_{i = 1}^2 \int_{\ell_i}\partial_{\nu}w\,d\mathcal{H}^1 + \sum_{i = 1}^2\int_{\phi_i}\partial_{\nu}w\,d\mathcal{H}^1.
\]
Notice that
\[
-\int_{\phi_i}\partial_{\nu}w\,d\mathcal{H}^1 = \mathcal{H}^1(\phi_i)\sqrt{h - \ga} \ge  |s_2 - s_1|\sqrt{h - \ga},
\]
while 
\[
\int_{\ell_i}\partial_{\nu}w\,d\mathcal{H}^1 \le \sup_G|\nabla w|\e_i.
\]
Consequently,
\[
2|s_2 - s_1|\sqrt{h - \ga} \le -\sum_{i = 1}^2\int_{\phi_i}\partial_{\nu}w\,d\mathcal{H}^1 = \sum_{i = 1}^2\int_{\ell_i}\partial_{\nu}w\,d\mathcal{H}^1 \le \sup_G|\nabla w|(\e_1 + \e_2),
\]
and the desired result readily follows.
\end{proof}

\subsection{Convergence of free boundaries for symmetric blow-up limits}
Throughout this subsection we will work under the assumptions of \Cref{bgl}. In particular, if $w$ is a blow-up limit with respect to the sequence $\{\rho_n\}_n$ of the symmetric local minimizer $u$, then the map $s \mapsto w(s,t)$ is increasing in $[0,\infty)$ (and decreasing in $(-\infty,0]$) for every $t \in \RR$. In turn, its free boundary restricted to the half-plane $\{s > 0\}$ can be described by the graph of a function $s = g_0(t)$, where $g_0 \colon \RR \to [0,\infty]$ is defined via
\begin{equation}
\label{g0DEF}
g_0(t) \coloneqq \inf\{s > 0 : w(s,t) > 0\}.
\end{equation}
We recall that by \Cref{gcont} the function $g_0$ is continuous in its effective domain. Furthermore, if we let
\begin{equation}
\label{gn}
g_n(t) \coloneqq \frac{g(\ga + \rho_n t) - g(\ga)}{\rho_n} = \frac{g(\ga + \rho_nt) + \frac{\la}{2}}{\rho_n}
\end{equation} 
for $g$ defined as in (\ref{gdef}), then we have that
\begin{equation}
\label{gninf}
g_n(t) = \inf\{s > 0 : u_n(s,t) > 0\}.
\end{equation}
Thus the free boundary of $u_n$ in $B_R(\bm{0}) \cap \{s > 0\}$ is described by the graph of $g_n$. It is then natural to ask whether $g_n$ converges to $g_0$.

\begin{lem}
\label{gntog0}
Let $g_n, g_0$ be given as above. Then for every $\tau \in \RR$ such that $g_0(\tau) < \infty$ we have that $g_0$ is finite in a neighborhood of $\tau$ and
\[
g_0(\tau) = \lim_{n \to \infty}g_n(\tau).
\]
\end{lem}
\begin{proof}
\textbf{Step 1:} Let $\tau$ be as in the statement. We begin by proving that either $g_0(t) < \infty$ for every $t < \tau$ or $g_0(t) < \infty$ for every $t > \tau$. Indeed, assume for the sake of contradiction that there exist $t_1 < \tau < t_2$ such that $g_0(t_1) = g_0(t_2) = \infty$, so that $w(s,t_1) = w(s,t_2) = 0$ for every $s > 0$ by (\ref{g0DEF}), and fix $s > g_0(\tau)$. For every $M > 0$, by the continuity of $w$, there exist $T_1, T_2 \in \RR$, $\e_1, \e_2 > 0$ such that
\[
t_1 \le T_i < \tau < T_i + \e_i \le t_2 
\]
and with the property that
\[
\{(s, T_1), (s,T_1 + \e_1 ), (s + M, T_2), (s + M, T_2 + \e)\} \subset \partial \{w > 0\}.
\]
Let $G$ be the region bounded by the free boundary $\partial \{w > 0\}$ and the two vertical line segments that connect the points $(s,T_1)$ with $(s,T_1 + \e_1)$ and $(s + M, T_2)$ with $(s + M, T_2 + \e_2)$. Then \Cref{nol} yields
\[
M \le \frac{C(t_2 - t_1)}{\sqrt{h - \ga}},
\]
a contradiction to the fact that $M$ is arbitrary. Hence $g_0(t) < \infty$ for all $t \le \tau$ or for all $t \ge \tau$. Without loss of generality, we assume the latter. Arguing by contradiction, assume that there exists a sequence $t_n \to \tau^-$ such that $g_0(t_n) = \infty$. Reasoning as above we see that necessarily $g_0(t) = \infty$ for $t_1 \le t < \tau$. In turn, since $w$ is continuous, it must be the case that $w(s,\tau) = 0$ for every $s > 0$, a contradiction to the assumption that $g_0(\tau) < \infty$.
\newline
\textbf{Step 2: } Suppose that there exists $\e > 0$ such that 
\begin{equation}
\label{limsupg0}
L \coloneqq \limsup_{n \to \infty}g_n(\tau) \ge g_0(\tau) + \e.
\end{equation}
By eventually extracting a subsequence we can assume that the limsup is achieved, and furthermore we notice that for every $n$ sufficiently large, (\ref{gninf}) and (\ref{limsupg0}) imply that $u_n(L - \e/2,\tau) = 0$. Since the map $s \mapsto u_n(s,\tau)$ is increasing by assumption, we have that $u_n(s,\tau) = 0$ for every $s \le L - \e/2$. In turn, passing to the limit in $n$, $w(s,\tau) = 0$ for every $s \le L - \e/2$, which is in contradiction with the definition of $g_0$ (see (\ref{g0DEF}) and (\ref{limsupg0})). This shows that 
\[
\limsup_{n \to \infty}g_n(\tau) \le g_0(\tau).
\]
Notice that if $g_0(\tau) = 0$ then there is nothing else to prove. Therefore, we can assume without loss that $g_0(\tau) > 0$. Assume for the sake of contradiction that for some $\e > 0$
\begin{equation}
\label{liminfg0}
\liminf_{n \to \infty}g_n(\tau) \le g_0(\tau) - 2\e.
\end{equation}
Since $g_0$ is continuous in a neighborhood of $\tau$, there exists $\delta = \delta(\e, \tau) > 0$ such that if $|t - \tau| < \delta$ then 
\[
g_0(\tau) - \e \le g_0(t).
\]
Notice that without loss of generality we can assume that $4\e < g_0(\tau)$. Fix $r < \min\{\e,\delta\}$ and set $\sigma \coloneqq g_0(\tau) - \e - r$.
Then $B_r(\sigma,\tau) \subset \{w = 0\}$ and thus it follows from \Cref{competition2} that 
\[
B_{r/2}(\sigma,\tau) \subset \{u_n = 0\}
\]
for every $n$ sufficiently large. In particular, $u_n(s,\tau) = 0$ for every $s \le \sigma + r/2$ and therefore
\begin{equation}
\label{gn52}
g_n(\tau) \ge \sigma + \frac{r}{2} \ge g_0(\tau) - \frac{3}{2}\e.
\end{equation}
Since (\ref{gn52}) is in contradiction with (\ref{liminfg0}) we conclude that
\[
\liminf_{n \to \infty}g_n(\tau) \ge g_0(\tau),
\]
which completes the proof.
\end{proof}

\section{A boundary monotonicity formula}
In this section we show that the boundary monotonicity formula of Weiss (see Theorem 3.3 and Corollary 3.4 in \cite{MR2047400}) holds at the point $\bm{x}_0 = (-\la/2,\ga)$ for local minimizers of $\J_h$ in $\K_{\ga}$ with bounded gradient in a neighborhood of $\bm{x}_0$. 

\begin{thm}
\label{bmf}
Given $m, h, \la > 0$ and $\ga < h$, let $u$ be a local minimizer of $\J_h$ in $\K_{\ga}$. Assume that there exist a constant $C > 0$ and $\mu < \min\{\ga, h- \ga, \la\}$ such that $|\nabla u(\bm{x})| \le C$ for all $\bm{x} \in B_{\mu}(\bm{x}_0)$. Then there exists $0 < r_0 < \mu$, depending on $C$ and $\e_0$, such that if for $r \in (0,r_0)$ we define
\begin{equation}
\label{Phidef}
\Phi(r) \coloneqq r^{-2}\int_{B_r(\bm{x}_0)}\left(|\nabla u|^2 + \chi_{\{u > 0\}}(h - y)\right)\,d\bm{x} - r^{-3}\int_{\partial B_r(\bm{x}_0)}u^2\,d\mathcal{H}^1,
\end{equation}
then for $\mathcal{L}^1$-a.e.\@ $\rho$ and $\sigma$ such that $0 < \rho < \sigma < r_0$,
\[
\Phi(\sigma) - \Phi(\rho) = \int_{\rho}^{\sigma}r^{-2}\int_{\partial B_r(\bm{x}_0)}2\left(\partial_{\nu}u - \frac{u}{r}\right)^2\,d\mathcal{H}^1dr - \int_{\rho}^{\sigma}r^{-3}\int_{B_r(\bm{x}_0)}\chi_{\{u > 0\}}(y - \ga)\,d\bm{x}dr.
\]
\end{thm}
\begin{proof}
\textbf{Step 1:} For simplicity we consider the translated function 
\[
w(\bm{x}) = u(\bm{x}_0 + \bm{x}).
\]
We begin by showing that for $\mathcal{L}^1$-a.e.\@ $r \in (0,r_0)$,
\begin{equation}
\label{limitidentity}
\int_{B_r(\bm{0})}\chi_{\{w > 0\}}(2h - 2\ga - 3y)\,d\bm{x} = \int_{\partial B_r(\bm{0})}r\left(|\nabla w|^2 - 2\left(\partial_{\nu}w\right)^2 + \chi_{\{w > 0\}}(h - \ga - y)\right)\,d\mathcal{H}^1.
\end{equation}
To this end, we consider the functional 
\[
\F_h(v) \coloneqq \int_{B_{r_0}(\bm{0})}\left(|\nabla v|^2 + \chi_{\{v > 0\}}(h - \ga - y)\right)\,d\bm{x},
\]
defined for $v \in \K(w, r_0)$ (see (\ref{Kw})). Notice that by the minimality of $u$ and our choice of $r_0$ the first variation of $\F_h$ with respect to domain variations vanishes at $w$. To be precise, for every $\bm{\phi} = (\phi_1,\phi_2) \in C^1(B_{r_0}(\bm{0}); \RR^2)$ which is compactly supported in $B_{r_0}(\bm{0}) \setminus \{(0,y) : y \ge 0\}$, if we set $w_{\e}(\bm{x}) \coloneqq w(\bm{x} + \e \bm{\phi}(\bm{x}))$ we have that $w_{\e} \in \K(w,r_0)$ for every  $\e$ sufficiently small and
\begin{align}
\label{domvar}
0 = &\ - \frac{d}{d\e}\F_h(w_{\e})_{|_{\e = 0}} \notag \\
= &\ \int_{B_{r_0}(\bm{0})}\Big(|\nabla w|^2\dive \bm{\phi} - 2\nabla w D\bm{\phi} \nabla w + \chi_{\{w > 0\}}(h - \ga - y) \dive \bm{\phi} - \chi_{\{w > 0\}}\phi_2 \Big)\,d\bm{x}.
\end{align}
For $r \in (0,r_0)$ and $\delta > 0$ define 
\begin{equation}
\label{etaxi}
\eta_\delta(\bm{x}) \coloneqq \max\left\{0,\min\left\{1, \frac{1}{\delta}(r - |\bm{x}|)\right\}\right\}, \quad \xi_\delta(\bm{x}) \coloneqq \min\left\{1,\frac{1}{\delta}\dist(\bm{x}, \{(0,y) : y \ge 0\})\right\}.
\end{equation}
Let $\bm{\phi}_\delta(\bm{x}) \coloneqq \eta_\delta(\bm{x}) \xi_\delta(\bm{x})\bm{x}$. By a standard density argument, for every $\delta > 0$ we can find a sequence $\{\bm{\phi}_{\delta,\e}\}_{\e}$ of functions in $C^1(B_{r_0}(\bm{0}); \RR^2)$ with compact support in $B_{r_0}(\bm{0}) \setminus \{(0,y) : y \ge 0\}$ such that $\bm{\phi}_{\delta,\e} \to \bm{\phi}_\delta$ in $W^{1,\infty}(B_{r_0}(\bm{0}),\RR^2)$. Using $\bm{\phi}_{\delta,\e}$ as test function in (\ref{domvar}), letting $\e \to 0$, and noticing that
\begin{align*}
D\bm{\phi}_\delta = &\ \eta_\delta \xi_\delta \id + \eta_\delta \nabla \xi_\delta \otimes \bm{x} + \xi_\delta \nabla \eta_\delta \otimes \bm{x}, \\
\dive \bm{\phi}_\delta = &\ 2 \eta_\delta \xi_\delta + \eta_\delta \nabla \xi_\delta \cdot \bm{x} + \xi_\delta \nabla \eta_\delta \cdot \bm{x}, \\
\end{align*}
we obtain the identity
\begin{equation}
\label{sumIk}
I_1^\delta + I_2^\delta + I_3^\delta = 0,
\end{equation}
where 
\begin{align*}
I_1^\delta \coloneqq &\, \int_{B_{r_0}(\bm{0})}\eta_\delta\xi_\delta\chi_{\{w > 0\}}(2h - 2\ga - 3y)\,d\bm{x}, \\
I_2^\delta \coloneqq &\, \int_{B_{r_0}(\bm{0})}\xi_\delta\Big(|\nabla w|^2\nabla \eta_\delta \cdot \bm{x} - 2 (\nabla w \cdot \bm{x})(\nabla w \cdot \nabla \eta_\delta) + \chi_{\{w > 0\}}(h - \ga - y)\nabla \eta_\delta \cdot \bm{x}\Big)\,d\bm{x}, \\
I_3^\delta \coloneqq &\, \int_{B_{r_0}(\bm{0})}\eta_\delta \Big(|\nabla w|^2\nabla \xi_\delta \cdot \bm{x} - 2 (\nabla w \cdot \bm{x})(\nabla w \cdot \nabla \xi_\delta) + \chi_{\{w > 0\}}(h - \ga - y) \nabla \xi_\delta \cdot \bm{x}\Big)\,d\bm{x}.
\end{align*}
By (\ref{etaxi}) and the monotone convergence theorem we have that
\begin{equation}
\label{I1klimit}
I_1^\delta \to \int_{B_r(\bm{0})}\chi_{\{w > 0\}}(2h - 2\ga - 3y)\,d\bm{x}.
\end{equation}
Observe that 
\[
\nabla \eta_\delta(\bm{x}) = 
\left\{
\arraycolsep=1.4pt\def\arraystretch{1.6}
\begin{array}{ll}
\displaystyle -\frac{\bm{x}}{\delta|\bm{x}|} & \text{ in } B_r(\bm{0}) \setminus B_{r - \delta}(\bm{0}), \\
0 & \text{ otherwise}.
\end{array}
\right.
\]
Thus we can rewrite $I_2^\delta$ as follows:
\begin{align*}
I_2^\delta = &\ - \frac{1}{\delta}\int_{B_r(\bm{0}) \setminus B_{r - \delta}(\bm{0})}\xi_\delta|\bm{x}|\left(|\nabla w|^2 - 2\left(\nabla w \cdot \frac{\bm{x}}{|\bm{x}|}\right)^2 + \chi_{\{w > 0\}}(h - \ga - y)\right)\,d\bm{x} \\
= &\ - \frac{1}{\delta}\int_{r - \delta}^r\int_{\partial B_s(\bm{0})}\xi_\delta s\left(|\nabla w|^2 - 2\left(\nabla w \cdot \frac{\bm{x}}{s}\right)^2 + \chi_{\{w > 0\}}(h - \ga - y)\right)\,d\mathcal{H}^1(\bm{x})ds.
\end{align*}
Consequently, by Fubini's theorem and Lebesgue's differentiation theorem, for $\mathcal{L}^1$-a.e.\@ $0 < r < r_0$, we have that
\begin{equation}
\label{I2klimit}
I_2^\delta \to - \int_{\partial B_r(\bm{0})}r\left(|\nabla w|^2 - 2\left(\partial_{\nu}w\right)^2 + \chi_{\{w > 0\}}(h - \ga - y) \right)\,d\mathcal{H}^1
\end{equation}
as $\delta \to 0^+$. By (\ref{sumIk}), (\ref{I1klimit}), and (\ref{I2klimit}), it follows that to conclude the proof of (\ref{limitidentity}) we are left to show that $I_3^\delta \to 0$ as $\delta \to 0^+$. To this end, we let
\begin{align*}
\Omega_\delta^+ \coloneqq &\, \left\{\bm{x} \in B_r(\bm{0}) : x > 0,\ y > 0, \text{ and } \dist(\bm{x}, \{(0,y) : y \ge 0\}) < \delta \right\}, \\
\Omega_\delta^- \coloneqq &\, \left\{\bm{x} \in B_r(\bm{0}) : x < 0,\ y > 0, \text{ and } \dist(\bm{x}, \{(0,y) : y \ge 0\}) < \delta \right\}, \\
\Omega_\delta^* \coloneqq &\, \left\{\bm{x} \in B_{\delta}(\bm{0}) : y < 0 \right\},
\end{align*}
and notice that
\begin{equation}
\label{gradxi}
\nabla \xi_\delta(\bm{x}) = 
\left\{
\arraycolsep=1.4pt\def\arraystretch{2}
\begin{array}{ll}
\displaystyle (\pm \delta^{-1},0) & \text{ in } \Omega_\delta^{\pm}, \\
\displaystyle \frac{\bm{x}}{\delta|\bm{x}|} & \text{ in } \Omega_\delta^*, \\
0 & \text{ otherwise}.
\end{array}
\right.
\end{equation}
From (\ref{gradxi}) and the fact that 
\[
\dist(\bm{x}, \{(0,y) : y \ge 0\}) = 
\left\{
\arraycolsep=1.4pt\def\arraystretch{2}
\begin{array}{ll}
\displaystyle |x| & \text{ in } \Omega_\delta^{\pm}, \\
\displaystyle |\bm{x}| & \text{ in } \Omega_\delta^*, \\
\end{array}
\right.
\]
we see that $|\nabla \xi_\delta \cdot \bm{x}| \le 1$ in $B_{r_0}(\bm{0})$, and consequently
\[
\left|\int_{B_{r_0}(\bm{0})}\eta_\delta \left(|\nabla w|^2 + \chi_{\{w > 0\}}(h - \ga - y)\right)\nabla \xi_\delta \cdot \bm{x}\,d\bm{x}\right| \le \int_{\{\nabla \xi_\delta \neq 0\}}\left(|\nabla w|^2 + \chi_{\{w > 0\}}(h - \ga - y)\right)\,d\bm{x}.
\]
Furthermore, the right-hand side in the previous inequality vanishes as $\delta \to 0^+$ by the dominated convergence theorem. It remains to show that 
\[
\int_{B_{r_0}(\bm{0})}\eta_\delta\left(x\partial_xw + y \partial_yw\right)\left(\nabla w \cdot \nabla \xi_\delta\right)\,d\bm{x} \to 0
\] 
as $\delta \to 0^+$. Since $|\nabla \xi_\delta||x| \le 1$, reasoning as above we see that 
\[
\left|\int_{B_{r_0}(\bm{0})}\eta_\delta x\partial_x w\left(\nabla w \cdot \nabla \xi_\delta\right)\,d\bm{x}\right| \le \int_{\{\nabla \xi_\delta \neq 0\}}|\partial_xw||\nabla w|\,d\bm{x} \to 0.
\]
Fix $\e \in (0,r)$. Using (\ref{gradxi}) and the fact that $\eta_\delta$ vanishes outside $B_r(\bm{0})$, we see that
\begin{align*}
\left|\int_{\Omega_\delta^+}\eta_\delta y \partial_yw(\nabla w \cdot \nabla \xi_\delta)\,d\bm{x}\right| \le &\ \frac{1}{\delta}\int_{(0,\delta) \times (0,r)}y|\partial_y w||\partial_xw|\,d\bm{x} \notag \\
\le &\ \frac{\e}{\delta}\int_{(0,\delta) \times (0,\e)}|\partial_y w||\partial_xw|\,d\bm{x} + \frac{r}{\delta}\int_{(0,\delta) \times (\e,r)}|\partial_y w||\partial_xw|\,d\bm{x}.
\end{align*}
Since $\nabla w$ is bounded, the first term on the right-hand side can be bounded uniformly in $\delta$, and so it vanishes as $\e \to 0^+$. By Theorem 1.1 in \cite{MR3916702}, we have that the extended free boundary 
\[
\overline{\partial \{w > 0\}} \cap \Omega_\delta^+ \setminus B_{\e/2}(\bm{0})
\]
is of class $C^{1,1/2}$. In turn, it follows from Corollary 8.36 in \cite{gilbargtrudinger} that 
\begin{equation}
\label{brw}
w \in C^{1,1/2}(\overline{\{w > 0\} \cap \Omega_\delta^+} \setminus B_{\e}(\bm{0})).
\end{equation}
In particular, this implies that $\nabla_{\tau}w = 0$ on $\left(\partial \{w > 0\} \cup \{(0,y) : y \ge 0\}\right) \setminus B_{\e}(\bm{0})$. Consequently, a change of variables and the dominated convergence theorem give
\[
\frac{r}{\delta}\int_{(0,\delta) \times (\e,r)}|\partial_y w||\partial_xw|\,d\bm{x} = r\int_{(0,1) \times (\e,r)}|\partial_yw(\delta x,y)||\partial_xw(\delta x,y)|\,d\bm{x} \to 0
\]
as $\delta \to 0^+$. Since similar estimates hold in $\Omega_\delta^-$ and $\Omega_\delta^*$, this concludes the proof of (\ref{limitidentity}).
\newline
\textbf{Step 2: } This step is dedicated to the proof of the integration by parts formula
\begin{equation}
\label{IBP}
\int_{B_r(\bm{0})}|\nabla w|^2\,d\bm{x} = \int_{\partial B_r(\bm{0})}w\partial_{\nu}w\,d\mathcal{H}^1,
\end{equation}
which holds for $\mathcal{L}^1$-a.e.\@ $r \in (0,r_0)$, and is in spirit very close to the result of Lemma 3.1 in \cite{MR1149865}. Let
\[
U_{\e,\eta} \coloneqq B_r(\bm{0}) \setminus \left(B_{\e}(\bm{0}) \cup \{\bm{x} : \dist(\bm{x}, \{(0,y) : y \ge 0\}) < \eta\} \right),
\]
and observe that by the divergence theorem, together with the fact that $w = 0$ on $\partial \{w > 0\}$,
\[
\int_{U_{\e,\eta} \cap \{w > 0\}}|\nabla w|^2\,d\bm{x} = \int_{\partial U_{\e,\eta} \cap \{w > 0\}}w\partial_{\nu}w\,d\mathcal{H}^1.
\]
Next, using the fact that $w$ is Lipschitz continuous in $B_{r_0}(\bm{0})$, that $w(0,y) = 0$ for $y > 0$, and (\ref{brw}), we obtain
\[
\lim_{\e \to 0^+} \lim_{\eta \to 0^+}\int_{U_{\e,\eta} \cap \{w > 0\}}|\nabla w|^2\,d\bm{x} = \lim_{\e \to 0^+} \int_{\partial(B_r(\bm{0}) \setminus B_{\e}(\bm{0}))}w\partial_{\nu}w\,d\mathcal{H}^1 = \int_{\partial B_r(\bm{0})}w\partial_{\nu}w\,d\mathcal{H}^1.
\]
Formula (\ref{IBP}) follows immediately upon noticing that 
\[
\int_{B_r(\bm{0})}|\nabla w|^2\,d\bm{x} = \lim_{\e \to 0^+} \lim_{\eta \to 0^+}\int_{U_{\e,\eta} \cap \{w > 0\}}|\nabla w|^2\,d\bm{x}.
\]
\newline
\textbf{Step 3: } By a direct computation we see that for $\mathcal{L}^1$-a.e.\@ $r \in (0,r_0)$,
\begin{align}
\label{Phi'}
\Phi'(r) = &\ -2r^{-3}\int_{B_r(\bm{0})}\left(|\nabla w|^2 + \chi_{\{w > 0\}}(h - \ga - y)\right)\,d\bm{x} \notag \\
&\ + r^{-2}\int_{\partial B_r(\bm{0})}\left(|\nabla w|^2 + \chi_{\{w > 0\}}(h - \ga - y) + 2r^{-2}w^2 - 2r^{-1}w\partial_{\nu}w \right)\,d\mathcal{H}^1,
\end{align}
where $\Phi$ is defined in (\ref{Phidef}) and we recall that $w(\bm{x}) = u(\bm{x}_0 + \bm{x})$. Moreover, by (\ref{limitidentity}) and (\ref{IBP}), we can rewrite (\ref{Phi'}) as
\[
\Phi'(r) = 2r^{-2}\int_{\partial B_r(\bm{0})}\left(\partial_{\nu}w - \frac{w}{r}\right)^2\,d\mathcal{H}^1 - r^{-3}\int_{B_r(\bm{0})}\chi_{\{w > 0\}}y\,d\bm{x},
\]
and the desired formula follows by integration.
\end{proof}

\begin{rmk} 
In view of \Cref{bgl}, the assumptions of \Cref{bmf} are satisfied by local minimizers which are symmetric in the sense of \Cref{sm}, provided $\bm{x}_0$ is an accumulation point for $\partial \{ u > 0\}$. Moreover, under the additional assumption that $\bm{x}_0$ is an isolated accumulation point for $\partial \{u > 0\}$ on $\partial \Omega$, the powerful regularity result of \cite{MR3916702} is not needed for the proof of \Cref{bmf}.
\end{rmk}

\begin{cor}
\label{Phi0}
Let $\Phi$ be defined as in \Cref{bmf}. Then $\Phi$ has finite right-limit at zero, i.e.,
\[
\lim_{\rho \to 0^+}\Phi(\rho) \eqqcolon \Phi(0^+) \in \RR.
\]
\end{cor}
\begin{proof}
Fix $0 < \sigma < r_0$ and consider $\rho < \sigma$. By \Cref{bmf}
\[
\Phi(\sigma) = \Phi(\rho) + A(\rho, \sigma) + B(\rho, \sigma) + C(\rho, \sigma),
\]
where 
\begin{align*}
A(\rho,\sigma) \coloneqq &\, \int_{\rho}^{\sigma}r^{-2} \int_{\partial B_r(\bm{x}_0)}2\left(\partial_{\nu}u(\bm{x}) - \frac{u(\bm{x})}{r}\right)^2\,d\mathcal{H}^1(\bm{x})dr, \\
B(\rho,\sigma) \coloneqq &\,  - \int_{\rho}^{\sigma}r^{-3}\int_{B_r(\bm{x}_0)}\chi_{\{u > 0\}}(y - \ga)\chi_{\{y \ge \ga\}}\,d\bm{x}dr, \\
C(\rho, \sigma) \coloneqq &\, - \int_{\rho}^{\sigma}r^{-3} \int_{B_r(\bm{x}_0)}\chi_{\{u > 0\}}(y - \ga)\chi_{\{y < \ga\}}\,d\bm{x}dr.
\end{align*}
Notice that the maps $\rho \mapsto A(\rho, \sigma)$ and $\rho \mapsto C(\rho, \sigma)$ are decreasing, while $r \mapsto B(\rho, \sigma)$ is increasing. Then
\begin{align*}
\lim_{\rho \to 0^+} A(\rho, \sigma) + C(\rho, \sigma) = &\, \sup\{A(\rho, \sigma) + C(\rho, \sigma) : 0 < \rho < \sigma\}, \\ 
\lim_{\rho \to 0^+} B(\rho, \sigma) = &\, \inf\{B(\rho, \sigma) : 0 < \rho < \sigma\} < \infty.
\end{align*}
In turn, $\Phi$ admits a limit as $\rho \to 0^+$ as it was claimed. Moreover, the fact that $|\Phi(0^+)| < \infty$ follows upon recalling that $u$ is Lipschitz continuous in a neighborhood of $\bm{x}_0$ and $u(\bm{x}_0) = 0$. Hence $u(\bm{x})^2 \le C|\bm{x} - \bm{x}_0|^2$, and so $\Phi$ is bounded (see (\ref{Phidef})).
\end{proof}

\begin{cor}
Under the assumptions of \Cref{bmf}, let $w$ be a blow-up limit of $u$ at $\bm{x}_0$ with respect to the sequence $\{\rho_n\}_n$. Then
\begin{equation}
\label{gradwzw}
\nabla w(\bm{z}) \cdot \bm{z} = w(\bm{z}) \quad \text{ for } \mathcal{L}^2\text{-a.e.\@ } \bm{z} \in \RR^2.
\end{equation}
\end{cor}
\begin{proof}
For every $r > 0$ and $n$ large enough so that $\rho_nr < r_0$, by the change of variables $\bm{x} = \bm{x}_0 + \rho_n\bm{z}$ we see that (\ref{Phidef}) can be rewritten as
\[
\Phi(\rho_nr) = r^{-2}\int_{B_r(\bm{0})}\left(|\nabla u_n|^2 + \chi_{\{u_n > 0\}}(h - \ga - \rho_nt)\right)\,d\bm{z} - r^{-3}\int_{\partial B_r(\bm{0})}u_n^2\,d\mathcal{H}^1,
\]
where the functions $u_n$ are defined as in (\ref{blowup}). Therefore, it follows from \Cref{bmf} that for $\mathcal{L}^1$-a.e.\@ $0 < R < S$ and $n$ large enough we have the formula
\[
\Phi(\rho_nS) - \Phi(\rho_nR) = \int_R^Sr^{-2}\int_{\partial B_r(\bm{0})}2\left(\partial_{\nu}u_n - \frac{u_n}{r}\right)^2\,d\mathcal{H}^1dr - \int_{R}^{S}r^{-3}\int_{B_r(\bm{0})}\chi_{\{u_n > 0\}}\rho_nt\,d\bm{x}dr.
\]
Letting $n \to \infty$, by \Cref{Phi0}, we obtain
\begin{align}
\label{posint}
0 = \lim_{n \to \infty} \Phi(\rho_nS) - \Phi(\rho_nR) = &\ \liminf_{n \to \infty} \int_R^Sr^{-2}\int_{\partial B_r(\bm{0})}2\left(\partial_{\nu}u_n - \frac{u_n}{r}\right)^2\,d\mathcal{H}^1dr \notag \\
\ge &\ \int_{B_S(\bm{0}) \setminus B_R(\bm{0})}2|\bm{z}|^{-4}\left(\nabla w(\bm{z}) \cdot \bm{z} - w(\bm{z})\right)^2\,d\bm{z}.
\end{align}
In turn, the integrand in (\ref{posint}) must be zero $\mathcal{L}^2$-a.e.\@ in $B_S(\bm{0}) \setminus B_R(\bm{0})$. By the arbitrariness of $R, S$, this concludes the proof.
\end{proof}

\section{Proof of \Cref{mainthm}}
In this section we present the proof of \Cref{mainthm}. The fundamental step in the proof is the following characterization of the possible blow-up limits at the point $\bm{x}_0$, defined as in \Cref{busec} (see (\ref{blowup}) and (\ref{blowupconv})).
\begin{thm}
\label{class}
Under the assumptions of \Cref{bgl}, let $w$ be a blow-up limit of $u$ at $\bm{x}_0$. Then either $w$ is identically equal to zero or $w(s,t) = (h - \ga)(-t)_+$.
\end{thm}
\begin{proof}
\textbf{Step 1:} We begin by showing that $w$ is a positively homogenous function of degree one. To see this let $\bm{z} \in \{w > 0\}$ and notice that
\[
\frac{d}{dt}\left(\frac{1}{t}w(t\bm{z})\right) = \frac{1}{t}\nabla w(t\bm{z})\cdot \bm{z} - \frac{1}{t^2}w(t\bm{z}) = \frac{1}{t^2}\left(\nabla w(t\bm{z}) \cdot t\bm{z} - w(t\bm{z})\right) = 0
\]
for every $t > 0$ such that $w(t\bm{z}) > 0$, where in the last equality we have used (\ref{gradwzw}). Consequently, it must be the case that $w(t\bm{z}) = tw(\bm{z})$ for every $t > 0$ such that $w(t\bm{z}) > 0$. From this we deduce that the entire ray $\{t \bm{z} : t \in \RR_+\}$ must necessarily be contained in $\{w > 0\}$. In particular, each connected component of $\{w > 0\}$ is a sector with vertex at the origin. Next, we claim that the opening angle of every such sector is $\pi$, i.e., that each connected component of $\{w > 0\}$ is a half-plane passing through the origin. To this end, we can find a rotation $R \in SO(2)$, a set of polar coordinates $(r,\theta)$, and a function $f$ in such a way that 
\[
f(r,\theta) = w(R(r\cos \theta, r\sin \theta)),
\]
and
\begin{equation}
\label{fharm}
\left\{
\arraycolsep=1.4pt\def\arraystretch{1.6}
\begin{array}{rll}
\Delta f = & 0 & \text{ in } S_{\alpha} \coloneqq \{(r,\theta) : 0 < r < \infty, 0 < \theta < \alpha\}, \\
f = & 0 & \text{ on } \partial S_{\alpha}.
\end{array}
\right.
\end{equation}
Notice that the homogeneity of $w$ implies that 
\[
f(r,\theta) = rf(1,\theta) = rh(\theta),
\]
for a function $h$ which satisfies
\[
h''(\theta) + h(\theta) = 0.
\]
In turn, $h(\theta) = c_1\cos \theta + c_2 \sin \theta$. Moreover, the boundary conditions in (\ref{fharm}) give that $c_1 = 0$ and $c_2\sin \alpha = 0$. Since $f > 0$ in $S_{\alpha}$, then it must be the case that $\alpha = \pi$. 
\newline
\textbf{Step 2:} Since $u$ is symmetric about the line $\{x = - \la/2\}$, then $u_n$, defined as in (\ref{blowup}), is symmetric about the $t$-axis, and so is $w$. This, together with the fact that $w(0,t) = 0$ for $t \ge 0$, shows that if $w$ is not identically equal to zero then either $w(s,t) = (h - \ga)(-t)_+$ or $w(s,t) = (h - \ga)|s|$. To conclude, it is enough to notice that $w(s,t) = (h - \ga)|s|$ does not minimize the functional $\F_h$ over the set $\K(w,1)$ (see (\ref{Jbu}) and (\ref{Kw})), since this would be in contradiction with \Cref{bumin}. 
\end{proof}

\begin{proof}[Proof of \Cref{mainthm}]
For $g$ defined as in (\ref{gdef}), assume for the sake of contradiction that 
\[
\liminf_{y \to \infty}\frac{|g(y) - g(\ga)|}{|y - \ga|} = \alpha < \infty,
\]
let $\{y_n\}_n$ be a sequence for which the limit is realized, and assume without loss of generality that $\{y_n\}_n$ is monotone. Let $\rho_n \coloneqq |y_n - \ga|$ and notice that for $n$ large enough
\[
\rho_n \le \sqrt{(g(y_n) - g(\ga))^2 + (y_n - \ga)^2} \le \beta|y_n - \ga| = \beta\rho_n, \quad \text{ where } \beta \coloneqq \sqrt{\alpha ^2 + 2}.
\]
In turn, \Cref{nondeg} gives that every blow up of $u$ at $\bm{x}_0$ with respect to the sequence $\{\rho_n\}_n$ is not identically equal to zero. Then, it follows from \Cref{class} that the half-plane solution
\begin{equation}
\label{onlybu}
w(s,t) = (h - \ga)(-t)_+
\end{equation}
is the unique blow-up limit. Assume first that $y_n \to \ga^+$, set $\rho_n \coloneqq y_n - \ga$ and let $u_n$ be defined as in (\ref{blowup}). Notice that by (\ref{gn})
\begin{equation}
\label{gn1toalpha}
0 \le g_n(1) = \frac{g(y_n) - g(\ga)}{y_n - \ga} = \frac{g(y_n) + \frac{\la}{2}}{y_n - \ga} \to \alpha.
\end{equation}
On the other hand, since $\{t \ge 0\} \subset \{w = 0\}$ by (\ref{onlybu}), it must be the case that $u_n \equiv 0$ in $B_{1/2}(\alpha + 1,1)$ by \Cref{fbchi}. This contradicts (\ref{gn1toalpha}).
Next, we assume that $y_n \to \ga^-$. Then $g_n(-1) \to \alpha$ and by the uniform convergence of $u_n$ to $w$ we see that
\[
0 = u_n(g_n(-1),-1) \to w(\alpha, -1) = h - \ga > 0.
\]
This concludes the proof.
\end{proof}

\section*{Acknowledgements}
This paper is part of the first author's Ph.\@ D.\@ thesis at Carnegie Mellon University. The authors acknowledge the Center for Nonlinear Analysis (NSF PIRE Grant No.\@ OISE-0967140) where part of this work was carried out. The research of G.\@ Gravina and G.\@ Leoni was partially funded by the National Science Foundation under Grants No.\@ DMS-1412095 and DMS-1714098. G.\@ Gravina also acknowledges the support of the research support programs of Charles University: PRIMUS/19/SCI/01 and UNCE/SCI/023. G.\@ Leoni would like to thank Ovidiu Savin and Eugen Varvaruca for their helpful insights. The authors would also like to thank Luis Caffarelli, Ming Chen, Craig Evans and Ian Tice for useful conversations on the subject of this paper.

\bibliographystyle{alpha}
\bibliography{waterwaves2}

\end{document}